\documentclass[reqno]{amsart}
\usepackage{amsmath}
\usepackage{amssymb}
\usepackage{amsfonts}
\usepackage{amscd}
\usepackage{amsthm}

\usepackage[ansinew]{inputenc}

\usepackage[english]{babel}

\usepackage{enumerate}
\usepackage{multirow}

\usepackage{graphicx}

\usepackage{url}
\usepackage[all]{xy}
\usepackage{hyperref}
\hypersetup{ 
  pdfauthor   = {Reynald Lercier, Christophe Ritzenthaler, Jeroen Sijsling}
  pdftitle    = {Explicit Galois obstruction and descent}
  pdfsubject  = {},
  pdfkeywords = {},
  backref=true, pagebackref=true, hyperindex=true, colorlinks=true,
  breaklinks=true, urlcolor=blue, linkcolor=blue, citecolor=blue, bookmarks=true,
  bookmarksopen=true}

\usepackage{cleveref}
\makeatletter
\newcommand\footnoteref[1]{\protected@xdef\@thefnmark{\ref{#1}}\@footnotemark}
\makeatother

\usepackage[ruled,vlined,boxed,linesnumbered,norelsize]{algorithm2e}

\allowdisplaybreaks

\theoremstyle{plain}
\newtheorem{theorem}{Theorem}[section]

\newtheorem{proposition}[theorem]{Proposition}
\newtheorem{lemma}[theorem]{Lemma}
\newtheorem{corollary}[theorem]{Corollary}

\theoremstyle{definition}
\newtheorem{definition}[theorem]{Definition}

\theoremstyle{remark}
\newtheorem{remark}[theorem]{Remark}
\newtheorem{example}[theorem]{Example}
\newtheorem{algo}[theorem]{Algorithm}

\makeatletter

\newcounter{myequation}[equation]

\numberwithin{equation}{section}

\def\ie{\textit{i.e.}}

\DeclareMathOperator{\Aut}{Aut}

\DeclareMathOperator{\Gal}{Gal}
\DeclareMathOperator{\GL}{GL}
\DeclareMathOperator{\Hom}{Hom}

\DeclareMathOperator{\PGL}{PGL}

\DeclareMathOperator{\SL}{SL}

\DeclareMathOperator{\Nm}{Nm}

\newcommand{\Bm}{\mathsf{B}}
\newcommand{\Cm}{\mathsf{C}}

\newcommand{\Qm}{\mathsf{Q}}

\newcommand{\Rm}{\mathsf{R}}

\newcommand{\Xm}{\mathsf{X}}
\newcommand{\Ym}{\mathsf{Y}}

\newcommand{\AAb}{\mathbb{A}}

\newcommand{\CC}{\mathbb{C}}

\newcommand{\NN}{\mathbb{N}}
\newcommand{\PP}{\mathbb{P}}
\newcommand{\QQ}{\mathbb{Q}}
\newcommand{\RR}{\mathbb{R}}

\newcommand{\ZZ}{\mathbb{Z}}

\newcommand{\CG}{\mathbf{C}}
\newcommand{\DG}{\mathbf{D}}

\newcommand{\dD}{\mathcal{D}}

\newcommand{\lL}{\mathcal{L}}
\newcommand{\mM}{\mathcal{M}}

\newcommand{\rR}{\mathcal{R}}
\newcommand{\sS}{\mathcal{S}}
\newcommand{\tT}{\mathcal{T}}
\newcommand{\uU}{\mathcal{U}}

\newcommand{\Gbar}{\overline{G}}

\newcommand{\smat}[4]{
  \left(
    \begin{smallmatrix}
      #1 & #2\\
      #3 & #4
    \end{smallmatrix}
  \right)
}

\title[Explicit Galois obstruction and descent]{Explicit Galois obstruction and
descent for hyperelliptic curves with tamely cyclic reduced automorphism group}

\author{Reynald Lercier}
\address{%
  \textsc{DGA MI}, %
  La Roche Marguerite, %
  35174 Bruz, %
  France. %
}
\address{%
  IRMAR, %
  Universit\'e de Rennes 1, %
  Campus de Beaulieu, %
  35042 Rennes, %
  France. %
} 
\email{reynald.lercier@m4x.org}

\author{Christophe Ritzenthaler}
\address{%
  IRMAR, %
  Universit\'e de Rennes 1, %
  Campus de Beaulieu, %
  35042 Rennes, %
  France. %
}
\email{ritzenthalerchristophe@gmail.com}

\author{Jeroen Sijsling} 
\address{Mathematics Institute, Zeeman Building, University of Warwick, %
Coventry CV4 7AL, United Kingdom.
}
\email{sijsling@gmail.com}

\thanks{The authors acknowledge support by grant ANR-09-BLAN-0020-01, and by
  the research programme \emph{Investissements d'avenir} (ANR-11-LABX-0020-01)
  of the Centre Henri Lebesgue. The third author was additionally supported by a
  Marie Curie Fellowship IEF-GA-2011-299887.}

\date{\today} 
\subjclass[2010]{14Q05 ; 13A50 ; 14H10 ; 14H25 ; 14H37}
\keywords{Hyperelliptic curve ; Galois descent ; field of definition ; field of
moduli ; invariants ; genus 3}

\begin{document}

\begin{abstract}
  This paper is devoted to the study of the Galois descent obstruction for
  hyperelliptic curves of arbitrary genus whose reduced automorphism groups are
  cyclic of order coprime to the characteristic of their ground field. We give
  an explicit and effectively computable description of this obstruction. Along
  the way, we obtain an arithmetic criterion for the existence of a so-called
  hyperelliptic descent.

  We define homogeneous dihedral invariants for general hyperelliptic curves,
  and show how the obstruction can be expressed in terms of these invariants. If
  this obstruction vanishes, then the homogeneous dihedral invariants can also
  be used to explicitly construct a model over the field of moduli of the curve;
  if not, then one still obtains a hyperelliptic model over a degree $2$
  extension of the field of moduli.
\end{abstract}

\maketitle

\section*{Introduction}

The classical problem of Galois descent, as first considered by Weil
in~\cite{wei56}, is the following:
Let $\Xm$ be a variety over the algebraically closure $K$ of a perfect base
field $k$.  Suppose that $\Xm$ is isomorphic with all its Galois conjugates
$\Xm^{\sigma}$ under the action of $\Gal(K | k)$, or in other words that $k$ is
the \emph{(Galois) field of moduli} of $\Xm$ for the extension $K | k$. Does
there then exist a model of $\Xm$ over $k$?

If such a model exists, then it is called a \emph{descent} of $\Xm$.
Generically, or more precisely, when the geometric automorphism group of $\Xm$
is trivial, there is no obstruction to descent~\cite{wei56}, but this
partial answer is unsatisfactory, as there are many interesting classes of
varieties with nontrivial automorphism group. This paper considers one such
class, namely that of hyperelliptic curves. The explicit form of their defining
equations makes hyperelliptic curves the simplest class of curves after conics
and elliptic curves (for which the answer to the descent question is
well-known to be affirmative). Due to the presence of the hyperelliptic
involution, hyperelliptic curves never have a trivial automorphism group. This
makes them a fundamental example in the study of the descent problem.

The problem in fact allows a further refinement for hyperelliptic curves;
instead of merely asking for some model over $k$, one can ask for a model that is
again given by a hyperelliptic equation $y^2 = p(x)$. Let us call such a model a
\emph{hyperelliptic descent}. Considering the homogenization of $p$ links the
study of hyperelliptic descent with the study of homogeneous binary forms.  This
is a great benefit, since not only was the invariant theory of these forms
extensively studied in the nineteenth century, but one can often also apply the
\emph{method of covariants}, as in~\cite{lrs}.

The answer to the descent question for hyperelliptic curves depends on the
\emph{reduced automorphism group} $\Gbar$ of $\Xm$, which is the quotient of the
automorphism group $G = \Aut (\Xm)$ of $\Xm$ by the hyperelliptic involution. We
assume that the characteristic of $k$ does not equal $2$ throughout this paper
in order to describe hyperelliptic curves as separable covers of a conic over
$k$. Under this running assumption, let us say that $\Gbar$ is \emph{tamely
cyclic} if it is cyclic of order coprime to the characteristic of $K$. Then
Huggins' seminal work~\cite{huggins} shows that if $\Gbar$ is not tamely cyclic,
then the curve $\Xm$ allows a hyperelliptic descent. For tamely cyclic $\Gbar$,
explicit counterexamples for descent were first constructed by
Earle~\cite{earle} and Shimura~\cite{shimura-moduli}. More recently, the full
classification of the hyperelliptic curves that do not allow a hyperelliptic
descent for the extension $\CC | \RR$  was initiated by
Bujalance-Turbek~\cite{bujtur} and completed by Huggins~\cite{huggins-thesis}.

In Section~\ref{sec:descent}, we give a complete answer to the descent problem
in the case where $\Xm$ is a hyperelliptic curve with tamely cyclic reduced
automorphism group, for any extension $K | k$. The problem is naturally
stratified by our notion of the \emph{type} of $\Xm$, which contains information
on the automorphism group and the Weierstrass points of $\Xm$. We refer to
Theorems~\ref{thm:MainTh3} and~\ref{thm:MainTh4} for precise statements, but
essentially, once the type is given, then either all the curves of that type
descend or the obstruction is classified by the solvability of a certain norm
equation.

If the descent obstruction vanishes, then Section~\ref{sec:descent} also shows
how a descent can be effectively constructed if $\Gbar$ is nontrivial. In
Section~\ref{sec:trivial}, we consider the slightly more involved case when
$\Gbar$ is trivial. In this case, efficient algorithms are constructed by using
the covariant method from~\cite{lrs}.  Finally, in Section~\ref{sec:counter}, we
show how to construct essentially all counterexamples to descent, which recovers
the aforementioned results on the extension $\CC | \RR$ as a special case. More
precisely, given any quadratic extension of fields $L | k$,
Theorem~\ref{thm:MainTh5} gives a completely explicit description of the
$K$-isomorphism classes of the curves which are defined over $L$ and
$K$-isomorphic with their conjugate, but that do not descend to $k$.

The norm equation mentioned above is in fact determined purely by the
\emph{homogeneous dihedral invariants} of the curve $\Xm$. These invariants,
which will be discussed in Section~\ref{sec:dihedral}, are closely related with
and indeed named after the dihedral invariants defined by Gutierrez and
Shaska~\cite{gutsha}. Like these invariants, they can be calculated quickly once
the curve $\Xm$ is given in standard form, a transformation to which can be
determined effectively by using the methods in~\cite[Sec.2]{lrs}. However, there
are a few important differences between our dihedral invariants and the original
ones in~\cite{gutsha}.

First of all, the models from which we derive our homogeneous dihedral
invariants are normalized in a weaker way than in~\cite{gutsha}. Second, the
homogeneous dihedral invariants give an effective approach to the reconstruction
and parametrization of forms with given invariants, also in the non-generic
cases where many of the coefficients in these normal forms are zero.  Third, and
contrary to what is suggested in Lemma~3.2 and Theorem~4.5 in~\cite{gutsha},
such non-generic reconstruction is in fact more involved than that in the
generic case.  Finally, the claimed reconstruction over the field of moduli $k$
in~\cite[Thm.4.5]{gutsha} actually takes place over a quadratic extension of
$k$, as was already pointed out in~\cite[Rem.4.17]{lr2012}. In
particular,~\cite[Cor.4.6]{gutsha} is incorrect, as can also be seen from the
results in~\cite[Sec.6]{huggins} and our complete classification of the
counterexamples in Theorem~\ref{thm:MainTh3}.

To describe our invariants, consider the subgroup $D$ of $\GL_2$ consisting of
diagonal and anti-diagonal matrices. Then the homogeneous dihedral invariants
are the invariants of binary forms under the action of the group $D \cap \SL_2
(K)$ that are moreover homogeneous as a function in the defining coefficients of
these forms. Alternatively, they are those $D \cap \SL_2 (K)$-invariant
polynomials in the defining coefficients for which the action by the diagonal
subgroup of $\GL_2$ is described by a character. All invariants for $D \cap
\SL_2 (K)$ (hence in particular the invariants for $D$ itself) can be expressed
as a rational function in the homogeneous dihedral invariants.

%For those
%curves of genus whose reduced automorphism group is tamely cyclic, the homogeneous
%dihedral invariants can conversely be written as expression in the homogeneous
%invariants for the action of the full group $\SL_2$.  In
%Section~\ref{sec:dihedral}, we perform these calculations explicitly for
%hyperelliptic curves of genus~$3$.

Before defining the homogeneous dihedral invariants and proving the main
theorem, we need a result relating the existence of a general descent with that
of a hyperelliptic descent. This theme is explored in Section~\ref{sec:hypdesc}.
Building on results by Mestre~\cite{mestre} and Huggins~\cite{huggins}, we shall
show in Theorem~\ref{thm:MainTh1} that these two variants of the descent
problems are in fact equivalent, except possibly when the genus $g$ of $\Xm$ is
odd and its reduced automorphism group is tamely cyclic of odd order. In this
latter case, Theorem~\ref{thm:MainTh4} shows that a descent always exists.
Furthermore, we completely classify the counterexamples to this equivalence in
this remaining case in Theorem~\ref{thm:MainTh5}.

Even more surprising is that the existence of a hyperelliptic descent of $\Xm$
turns out to allow an arithmetic characterization. To formulate this result,
consider the quotient $B = \Xm / G$ of $\Xm$ by its full automorphism group. The
curve $B$ has a canonical descent $B_0$ to $k$, and it is well-known (see for
example \cite[Cor.2.3]{debes-emsalem}) that the presence of a
point of $B_0$ over the field of moduli is a \emph{sufficient} condition for
\emph{some} descent of $\Xm$ to exist. In Theorem~\ref{thm:MainTh2}, we show
that in fact the existence of such a rational point is \emph{equivalent} with
the existence of a \emph{hyperelliptic} descent of $\Xm$. In particular, we see
that $\Xm$ always admits a hyperelliptic equation over a degree $2$ extension of
$k$.

These results simplify matters from a theoretical point of view. The more
general obstruction criterion in~\cite[Sec.4]{debes-emsalem} describes the
descent obstruction in terms of the triviality of one of \emph{infinitely} many
element of $H^2$-cohomology groups. For hyperelliptic curves, the descent
obstruction turns out to be equivalent to the triviality of a single twist
(namely $B_0$) of $\PP^1_k$. Alternatively, this amounts to the triviality of a
\emph{single} element of an $H^1$-cohomology group. It is this pleasant surprise
that makes the theory of Galois descent for hyperelliptic curves both
conceptually simple and effectively computable.

After the proof of the main Theorems~\ref{thm:MainTh3} and~\ref{thm:MainTh4}, we
turn to algorithmic considerations and the implementation of our results in
Section~\ref{sec:imp}. Our \texttt{Magma}~\cite{Magma} functionality is
available
online\footnote{\label{foothyp}\url{http://iml.univ-mrs.fr/~ritzenth/programme/hyp-desc.tgz}}.
We also discuss how this implementation can be combined with the results
of~\cite{lr2012}. This concludes the exploration of the arithmetic aspects of
the moduli space of hyperelliptic genus $3$ curves started in that article; it
shows how to reconstruct any given genus $3$ curve from its invariants over an
extension of the field of moduli of minimal degree (which we now know to be at
most $2$).  This additional functionality has been added to the package
\texttt{g3twists}\footnote{\label{footg3}\url{http://iml.univ-mrs.fr/~ritzenth/programme/g3twists_v1.1.tgz}},
and is included in the current versions of \texttt{Magma}.
Section~\ref{sec:conc} concludes the paper and briefly discusses the remaining
open questions on the descent of hyperelliptic curves.

Table~\ref{tab:issues} gathers our state of knowledge (we \textbf{emphasize}
what is proved in the present paper).
\begin{table}[htbp]
  \centering
  \renewcommand{\arraystretch}{1.1}
  {\small
    \begin{tabular}{|c|c|c|c|c|c|}
      \hline
      $\Gbar$ &%
      Condition &%
      \begin{tabular}{c} Descent $\Leftrightarrow$\\
        Hyperelliptic \\ descent\end{tabular} &%
      \begin{tabular}{c} Obstruction\\
        to descent\end{tabular} &%
      \begin{tabular}{c} Effective\\
        Method\end{tabular} \\ \hline\hline
      \begin{tabular}{c} Not tamely \\ cyclic \end{tabular} &%
      - &%
      \begin{tabular}{c} Yes\\
        %{\footnotesize \cite{huggins-thesis}}) \end{tabular} &%
        \cite{huggins-thesis} \end{tabular} &%
      \begin{tabular}{c} No\\
        \cite{huggins-thesis} \end{tabular} &%
    ?%
      \\\hline%
      \multirow{2}{*}{\begin{tabular}{c} Tamely cyclic \\ and $\# \Gbar > 1$
      \end{tabular}} &%
      \begin{tabular}{c} $g$ odd and \\
      \#$\Gbar$ odd\end{tabular}&%
      \begin{tabular}{c} \textbf{No} \\
      Ex.\ \ref{ex:condesc} \end{tabular}&%
      \begin{tabular}{c} \textbf{No} \\
      Thm.\ \ref{thm:MainTh4} \end{tabular}&%
      \begin{tabular}{c} \textbf{Yes} \\
      Alg.\ \ref{alg:nonhypdesc} \end{tabular}%
      \\
      \cline{2-5}%
      &%
      \begin{tabular}{c} $g$ even or \\
      \#$\Gbar$ even \end{tabular}&%
      \begin{tabular}{c} \textbf{Yes} \\
      Thm.\ \ref{thm:MainTh1} \end{tabular}&%
      \begin{tabular}{c} \textbf{Yes} \\
      Thm.\ \ref{thm:MainTh3} \end{tabular}&%
      \begin{tabular}{c} \textbf{Yes} \\
      Alg.\ \ref{alg:hypdesc} \end{tabular}%

      \\\hline%

      \multirow{2}{*}{$\# \Gbar > 1$} &%
      $g$ odd &%
      \begin{tabular}{c} No\\
        \cite{lr2012} \end{tabular} &%&%
      \begin{tabular}{c} No\\
        \cite{lr2012} \end{tabular} &%&%
      \begin{tabular}{c}
        \textbf{Generic} if $g \leq 2^7$ \\
        Rem.\ \ref{rem:nontrivcov} \end{tabular} %&%
      \\\cline{2-5}%
       &%
      $g$ even &%
      \begin{tabular}{c} Yes\\
        \cite{mestre} \end{tabular} &%&%
      \begin{tabular}{c} Yes\\
        \cite{mestre} \end{tabular} &%&%
      \begin{tabular}{c}
        \textbf{Generic} if $g \leq 2^7$ \\
        Rem.\ \ref{rem:nontrivcov} \end{tabular} %&%
      \\%

      \hline \hline
    \end{tabular} 
  }\medskip
  \caption{Issues addressed in the present paper.\label{tab:issues}
  } 
\end{table}

\subsection*{Notation}

We let $k$ be a perfect field of odd characteristic, and we let $K$ denote its
algebraic closure. We denote $\Gamma = \Gal (K | k)$. The curves over $K$ and
its subfields that are considered in this paper will be smooth, proper and
geometrically irreducible throughout. We define a \emph{hyperelliptic curve}
over $k$ as in~\cite[Sec.1.2]{lr2012}; that is, a curve $\Cm$ over $k$ is
hyperelliptic if and only if it admits a degree $2$ morphism to $\PP^1_K$ over
$K$. In this case, $\Cm$ admits a unique corresponding~\emph{hyperelliptic
involution} $\iota$ for which the quotient $\Cm / \iota$ is a conic over $k$.

In what follows, $\Xm$ denotes a hyperelliptic curve over $K$ of genus $g$ whose
field of moduli with respect to the extension $K | k$ equals $k$.
% We slightly expand the usual terminology; we call a curve over a field $L$
% arithmetically hyperelliptic if it admits a morphism to $\PP^1$ defined over
% its base field.  Therefore $2:1$ covers of pointless conics are hyperelliptic
% but not arithmetically hyperelliptic.
We denote the group $\Aut (\Xm)$ of automorphisms of $\Xm$ defined over $K$ by
$G$. The \emph{reduced automorphism} group $\Gbar = G / \iota$ is the quotient
of $G$ by the central element $\iota$. In the second half of this paper, we will
additionally suppose that $\Gbar$ is \emph{tamely cyclic}, \ie, cyclic of order
coprime to the characteristic of $K$. Finally, given a curve $\Xm$ and a divisor
$\dD$, we denote the group of automorphisms $\alpha$ of $\Xm$ over $K$ such that
the pushforward $\alpha_* (\dD)$ equals $\dD$ by $\Aut (\Xm,\dD)$.

We will occasionally construct a model of $\Xm$ over an intermediate field $k
\subseteq L \subseteq K$. When considering such curves over intermediate fields,
we restrict our consideration of morphisms to those defined over $L$, unless
explicitly specified otherwise. We denote the corresponding automorphism groups
by $\Aut_L (\Xm)$, \textit{et cetera}.

If $\varphi : X \to Y$ is a morphism between algebraic curves, then the
\emph{ramification divisor} of $\varphi$ is the divisor of the points on $X$
that ramify under $\varphi$. The \emph{branch divisor} of $\varphi$ is the image
of this divisor under $\varphi_*$. Note that we use these divisors without
multiplicities throughout.

We adopt the usual notation of denoting the Galois action by a superscript,
\emph{e.g.} $f^{\sigma}$ for the conjugation on a binary form. We consider this
as a left action, which leads to the somewhat counterintuitive equality
$f^{\sigma \tau} = (f^{\tau})^{\sigma}$.

We use the notation $\CG_n$ (resp.\ $\DG_{2 n}$) for the cyclic group with $n$
elements (resp.\ the dihedral group with $2 n$ elements. Given two homogeneous
binary forms $f_1$ and $f_2$ over a subfield $L$ of $K$, we say that $f_1 \sim
f_2$ if there exists a $\lambda$ in $L^*$ such that $f_1 = \lambda \cdot f_2$.
Given a matrix $A = \smat{a}{b}{c}{d}$ over $K$, we let $A . f$ be the
polynomial given by $(A . f) (x,z) = f ( A^{-1} (x,z))$. Finally, given a binary
form $f$ over $K$, we denote by $\Aut (f)$ the group of matrices $A$ up to
scalar in $\PGL_2 (K)$ such that $A . f \sim f$.

By $\zeta_n$, we denote a fixed choice of $n$-th root of unity in $K$; these
roots are chosen in such a way to respect the standard compatibility conditions
when raising to powers. The cyclic groups $\CG_n = \ZZ / n \ZZ$ are
always considered as being embedded in $\PGL_2 (K)$, by sending the generator
$1$ of $\CG_n$ to the automorphism acting by $(x:z) \mapsto (\zeta_n x
: z)$.

Throughout, we usually denote objects that are defined over the ground field $k$
by a zero-subscript, so that for example $X_0$ is typically a hyperelliptic
curve over $k$.

\section{Descent and hyperelliptic descent}\label{sec:hypdesc}

Consider a perfect field $k$ of odd characteristic, and let $K$ be the algebraic
closure of $k$. Let $\Gamma = \Gal (K | k)$, and let $\Cm$ be a curve over $K$.

\begin{definition}\label{def:fom}
  The \emph{(Galois) field of moduli} of $\Cm$ with respect to the extension $K |
  k$ is the fixed field of the group $\{ \sigma \in \Gamma : \Cm \; \text{is
  isomorphic to} \; \Cm^{\sigma} \; \mathrm{over} \; K \}$.
\end{definition}

\begin{remark}\label{rem:insep}
  For more general extensions $K | k$, one usually defines the field of moduli
  of $\Cm$ with respect to $K | k$ as the intersection of all fields of
  definition of $\Cm$ that are contained in $K$. The Galois field of moduli in
  the previous definition is then a purely inseparable extension of this more
  general field of moduli by~\cite{Sekiguchi}. We refer to
  Section~\ref{sec:conc} for some open questions concerning these matters.
\end{remark}

\begin{definition}\label{def:model}
  Let $L \subset K$ be a subfield of $K$ containing $k$. A \emph{model} of $\Cm$
  over $L$ is a curve $\Cm_0$ over $L$ such that $\Cm$ is isomorphic to $\Cm_0$
  over $K$. The field $L$ is then called a \emph{field of definition} for $\Cm$.
  
  A model of $\Cm$ over its field of moduli is called a \emph{descent} of $\Cm$.
  If such a model exists, then $\Cm$ is said to \emph{descend} (to its field of
  moduli). If not, then we say that there is \emph{descent obstruction} for $\Cm$.
\end{definition}

%In the upcoming Theorem~\ref{thm:MainTh2} we will quantify the descent
%obstruction in Definition~\ref{def:fom} (but see also~\cite{debes-emsalem} for
%a general description in terms of $H^2$-cohomology groups).
%
%Note that once a curve $C$ is given, we may always assume that the field of
%moduli of $C$ for the extension $K|k$ equals $k$, namely by extending the base
%field $k$ to the field of moduli.

For hyperelliptic curves, one can ask for a more specific form of descent.

\begin{definition}\label{def:hypdesc}
  Let $\Xm$ be a hyperelliptic curve over $K$ of genus $g$ whose field of
  moduli for the extension $K | k$ equals $k$. A \emph{hyperelliptic descent} of
  $\Xm$ is a model $\Xm_0$ of $\Xm$ over $k$ that is defined by a homogeneous
  polynomial $f_0 (x,z)$ of degree $2 g + 2$ over $k$ without repeated roots.
  More precisely, this is to say that $\Xm_0$ is the desingularization of the
  curve $y^2 = f_0(x,z)$ in the $(1,1,g+1)$-weighted projective $(x,z,y)$-space
  over $k$. 
\end{definition}

\begin{remark}
  There is a slight ambiguity to be noted. According to
  Definition~\ref{def:hypdesc}, any descent $\Xm_0$ of a hyperelliptic curve
  $\Xm$ is in fact hyperelliptic as a curve over $k$. However, such a descent is
  not always a hyperelliptic descent; this is the case if and only if the
  quotient $\Qm_0$ of $\Xm_0$ by its hyperelliptic involution $\iota_0$ is
  isomorphic to $\PP^1$ over $k$.
\end{remark}

\subsection{Equivalence between descent and hyperelliptic descent}

%By the proof of~\cite[Cor.1.6.6]{huggins}, a hyperelliptic descent to $k$ for
%$\Xm$ exists if $k$ is finite. We therefore assume it to be infinite throughout
%this section, mainly in order to be able to apply Lemma~\ref{lem:RCons}.

A fundamental result of Mestre~\cite{mestre} tells us that if $g$ is even, then
the curve $\Xm$ descends if and only if it descends hyperelliptically.  However,
when $g \geq 3$ is odd, this need not be the case. A counterexample is given in
the discussion after~\cite[Prop.4.13]{lr2012}. Due to the simpler nature of
hyperelliptic descent, we now first study in which other cases the equivalence
indicated by Mestre continues to hold.

It turns out that the answer to this question depends on the reduced
automorphism group $\Gbar = \Aut (\Xm) / \iota$. To get an idea of the problem,
we first consider the case of trivial $\Gbar$.  As in~\cite[Sec.4.3]{lr2012},
one shows that degree~$2$ covers of pointless conics over $k$ whose branch locus
is Galois stable give rise to curves over $K$ that have trivial reduced
automorphism group and nontrivial descent obstruction. Therefore in this case
there exist curves that descend but do not descend hyperelliptically. Number
fields are an important and naturally occurring class of fields over which such
covers of conics exist.

This previous paragraph can be seen as one of the few exceptions to the main
statement of the following Theorem, which we will prove in this Section, and
which we will later use to give an arithmetic criterion for the existence of a hyperelliptic descent in Theorem~\ref{thm:MainTh2}.

\begin{theorem}\label{thm:MainTh1}
  Let $\Xm$ be a hyperelliptic curve over $K$ of genus $g$ whose field of moduli
  for the extension $K | k$ equals $k$. Let $\Gbar$ be the reduced automorphism
  group of $\Xm$. Then the existence of a descent of $\Xm$ is equivalent to the
  existence of a hyperelliptic descent, except possibly when $g$ is odd and
  $\Gbar$ is tamely cyclic of odd cardinality.
\end{theorem}

\begin{remark}
  In the remaining case where $g$ and $\# \Gbar$ are both odd, we refer the
  reader to Theorem~\ref{thm:MainTh4} for a proof that $\Xm$ always descends.
  Moreover, in Example~\ref{ex:condesc} we will explicitly construct a
  hyperelliptic curve with nontrivial reduced automorphism group and field of
  moduli $\QQ$ that descends, but does not descend hyperelliptically.
\end{remark}
  
\begin{remark}
  An arithmetic criterion for the existence of a hyperelliptic descent is
  given in Theorem~\ref{thm:MainTh2}.
\end{remark}

In order to prove Theorem~\ref{thm:MainTh1}, we will first need a few technical
lemmata to deal with the case where the reduced automorphism group contains an
element of order $2$. The first such lemma is Lemma~3.1
from~\cite{xarles-towers}, there attributed to Poonen and to Witt before him.
Here we give a stronger version of this result.

\begin{lemma}\label{lem:xarles}
  Let $f : \Qm_0 \to \Bm_0$ be a nonconstant morphism between genus $0$
  curves over $k$ of degree $n$.
  \begin{enumerate}
    \item If $n$ is even, then $\Bm_0$ is isomorphic with $\PP^1$ over $k$.
    \item If $n$ is odd, then $\Bm_0$ is isomorphic with $\Qm_0$ over $k$.
  \end{enumerate}
\end{lemma}
\begin{proof}
  As in the proof of ~\cite[Lem.3.1]{xarles-towers}, one shows that the class of
  $\Qm_0$ in the Brauer group of $k$ is $n$ times that of $\Bm_0$. The result
  then follows from the fact that these classes are $2$-torsion elements.
\end{proof}

We will now apply Lemma~\ref{lem:xarles} in the situation of interest to us. In
what follows, our frequent hypothesis that $\Qm_0$ not be isomorphic with
$\PP^1$ is not always necessary, but we will only need the lemmata in this case.
Moreover, our current exposition allows for a more unified treatment of finite
and infinite base fields $k$.

\begin{lemma}\label{lem:RCons}
  Let $\Qm_0$ be a genus $0$ curve over $k$ that is not isomorphic with $\PP^1$
  over $k$, and let $\alpha_0 \in \Aut_k (\Qm_0)$ be an automorphism of order
  $2$ of $\Qm_0$ that is defined over $k$.  Then there exists a $k$-rational
  divisor $\rR_0$ of degree $6$ on $\Qm_0$ such that $\Aut_{K} (\Qm_0,\rR_0)$ is
  generated by $\alpha_0$. 
\end{lemma}
\begin{proof}
%  We consider three $k$-rational points $p_1 , p_2 , p_3$ on $\Qm / \alpha$ that
%  are not in the branch locus of the canonical map $\pi : \Qm \to \Qm /
%  \alpha$.  Using a geometric argument, we will show that for a sufficiently
%  generic choice of $p_1 , p_2 , p_3$ the divisor $\rR = \pi^{-1} (p_1 + p_2 +
%  p_3)$ satisfies $\Aut_{K} (\Qm,\rR) = \left\langle \alpha \right\rangle$.
%
  Consider the morphism from the affine space $\AAb^3$ to the moduli space
  $\mM_2$ of genus $2$ curves that sends a triple $(\lambda,\mu,\nu)$ to the
  curve $\Xm_{\lambda,\mu,\nu}$ given by the hyperelliptic equation
  $y^2=(x^2-\lambda)(x^2-\mu)(x^2-\nu)$. The results in~\cite{carquer} show that
  the locus $\lL$ of $\AAb^3$ for which the reduced automorphism group of
  $\Xm_{\lambda,\mu,\nu}$ is strictly larger than $\CG_2$ is of codimension~$1$
  in $\AAb^3$.
  %We can pull back the divisor $\dD$ by the induced isomorphism $\varphi :
  %\Qm/\alpha \to \PP^1_K$ to the arithmetic situation $\Qm \to \Qm/\alpha$.

  Now let $\pi_0 : \Qm_0 \to \Qm_0 / \alpha_0$ be the quotient morphism. We
  choose coordinates over $K$, that is to say, $K$-isomorphisms $\varphi : \Qm_0
  \to \PP^1$ and $\psi : \Qm_0 / \alpha_0 \to \PP^1$. Using the
  three-transitivity of $\Aut_K (\PP^1)$, we see that we can do this in such a
  way that the coordinatization $\psi \pi_0 \varphi^{-1}$ of the projection
  $\pi$ is given by the degree~$2$ map $(x:z) \mapsto (x^2:z^2)$. Let $q,r,s$ be
  three points on $\Qm_0 / \alpha_0$ that are not branch points of $\pi_0$. Then
  under our coordinatization, and considering $\AAb^1$ as a subvariety of
  $\PP^1$ via the coordinate $t = x / z$, the divisor $\rR = \pi_0^{-1} (r + s +
  t)$ on $\Qm_0$ is isomorphic over $K$ to the divisor $\psi^* (r) + \psi^* (s)
  + \psi^* (t)$. Therefore the hyperelliptic curve defined by taking a degree
  $2$ cover of $\Qm_0$ ramified over $\rR$ is isomorphic to the curve
  $\Xm_{\psi(r), \psi (s), \psi (t)}$.

  %This
  %degree~$2$ map is branched at $(0:1)$ and $(1:0)$.
  %, and we can moreover identify $q_1 = (1:1)$ after the choice of a point
  %$q_1$ of $\Qm$ giving rise to a rational point $p_1$ on $\Qm / \alpha$. 
  %Choose three distinct points $q_1=(\lambda:1)$, $q_2 = (\mu:1)$ and
  %$q_3=(\nu:1)$ with $\lambda,\mu,\nu \ne 0,\infty$. 
  %and $\mu \ne 0,1$ and 

  The transformation $\psi^{-1} (\lL)$ of the exceptional locus $\lL \subset
  \AAb^3$ is a codimension $1$ locus in $(\Qm_0 / \alpha_0)^3$. Note that $k$ is
  infinite by the existence of $\Qm_0$. This implies that the set of
  $k$-rational points is dense in $\Qm_0 / \alpha_0$, which is isomorphic with
  $\PP^1_k$ by Lemma~\ref{lem:xarles}(i).  Therefore we can find a rational
  point $(r_0 , s_0 , t_0)$ of $(\Qm_0 / \alpha_0)^3$ outside the exceptional
  locus. By construction, the divisor $\rR_0 = \pi^{-1} (r_0 + s_0 + t_0)$ now
  satisfies our requirements.
\end{proof}

%\begin{remark}
  %Define the \emph{splitting field} of a divisor to be the extension generated
  %by the coordinates of its points. Intuitively, it seems clear that in the
  %proposition above, we can choose the splitting field of $R$ to be a given
  %quadratic extension of $k$, but finding a rigorous proof of this appears to
  %be tricky.
%\end{remark}

\begin{lemma}\label{lem:QuadExt}
  Let $\Qm_0$ be a genus $0$ curve over $k$ that is not isomorphic with $\PP^1$
  over $k$, and let $\alpha_0$ be an automorphism of order $2$ of $\Qm_0$ that
  is defined over $k$. 
%Then we have the following.
%\begin{enumerate}[(i)]
%\item 
  Then there exists a quadratic extension $L$ of $k$ and an isomorphism $\varphi
  : \Qm_0 \to \PP^1$ over $L$ such that $\varphi^{\sigma} = \varphi \alpha$ for
  the generator $\sigma$ of $\Gal (L |k)$.
%\item If $\Qm \cong \PP^1_k$, then the previous conclusion holds if and only if
%$\alpha$ is represented by a matrix $M$ over $k$ for which $- \det (M)$ is a
%norm from a quadratic extension of $k$. In particular, this is true if $k$ is a
%number field.
%\end{enumerate} 
\end{lemma}
\begin{proof}
%First 
  %Suppose that $\Qm$ is not isomorphic to $\PP^1$ over $k$.
  Choose $\rR_0$ as in Lemma~\ref{lem:RCons} and consider the pair
  $(\Qm_0,\rR_0)$, which is defined over $k$. Over $K$, there exists a degree
  $2$ cover $\Xm$ of $\Qm_0$ branched in $\rR_0$, which has reduced geometric
  automorphism group $\CG_2$. We emphasize that \emph{a priori} the cover $\Xm$
  need not be defined over $k$, even though $(\Qm_0,\rR_0)$ is.

  Regardless, the field of moduli of $\Xm$ with respect to the extension $K | k$
  equals $k$. Indeed, the configuration $(\Qm_0,\rR_0)$, which determines the
  isomorphism class of $\Xm$ over $K$, is Galois stable. Alternatively, if we
  choose some $K$-isomorphism $i : \Qm_0 \to \PP^1$, then we have $(i_*
  (\rR_0))^{\sigma} = i_*^{\sigma} (\rR_0^{\sigma}) = i_*^{\sigma} (\rR_0)$ for
  $\sigma \in \Gamma$. This shows that $i_* (\rR_0)$, which is the branch locus
  of $\Xm$, and $(i_* (\rR_0))^{\sigma}$, which is the branch locus of
  $\Xm_0^{\sigma}$,  differ by the $K$-automorphism $i_*^{\sigma} i_*^{-1}$ of
  $\PP^1$. We see that the branch loci of $\Xm_0$ and $\Xm_0^{\sigma}$,
  considered as degree $2$ covers, can be transformed into one another over $K$.
  The hyperelliptic curves $\Xm_0$ and $\Xm_0^{\sigma}$ are therefore
  $K$-isomorphic.

  By~\cite[Thm.6]{carquer}, this implies that the genus $2$ curve $\Xm$ is
  hyperelliptically defined over $k$. The descent morphism $\Xm \to \Xm_0$ to a
  model $\Xm_0$ over $k$ then yields an isomorphism
  \begin{equation*}
    \varphi : (\Qm_0,\rR_0) \longrightarrow (\PP^1 , \sS_0)
  \end{equation*}
  over some Galois extension $M$ of $k$. Then the map $\Gal (M | k) \to \Aut_{K}
  (\Qm_0,\rR_0) = \left\langle \alpha \right\rangle$ that sends $\tau$ to
  $\varphi^{-1} \varphi^{\tau}$ is a homomorphism because $\Aut_{K}
  (\Qm_0,\rR_0) = \Aut_k (\Qm_0,\rR_0)$. Indeed, we have $\varphi^{-1}
  \varphi^{\tau_1 \tau_2} = \varphi^{-1} \varphi^{\tau_1}
  (\varphi^{-1})^{\tau_1} \varphi^{\tau_1 \tau_2} = \varphi^{-1}
  \varphi^{\tau_1} \varphi^{-1} \varphi^{\tau_2}$.

  The kernel of this homomorphism is not all of $\Gal (M | k)$, because that
  would imply that $\Qm_0$ is isomorphic to $\PP^1$ over $k$. So this kernel cuts
  out a quadratic extension $L$ of $k$. By construction, $\varphi$ is then
  defined over $L$, and we have that $\varphi^{\sigma} = \varphi\, \alpha_0$.
\end{proof}

%\begin{remark}
  %Note that finding an $M$ as in the final part of the proposition is
  %impossible for the extension $\CC | \RR$ if $\nu < 0$.
%\end{remark}

\begin{proposition}\label{prop:order2}
  Let $\Xm_0$ be a hyperelliptic curve over $k$. Suppose that the reduced
  automorphism group $\Gbar$ of $\Xm$ contains an element $\alpha_0$ of order
  $2$ that is defined over $k$. Then $\Xm_0$, considered as a curve over $K$,
  descends hyperelliptically to $k$.
\end{proposition}
\begin{proof}
  Let $\iota_0$ be the hyperelliptic involution of $\Xm_0$. Then $\iota_0$ is
  defined over $k$, because it is the unique involution of $\Xm_0$ for which the
  quotient $\Qm_0 = \Xm_0 / \iota_0$ is of genus $0$. Consider $\Qm_0$ as a
  curve over $k$. If $\Qm_0$ is isomorphic to $\PP^1$ over $k$, then we are
  done. So assume the contrary.

  Let $\rR_0$ be the branch locus of the quotient morphism $\Xm_0 \to \Qm_0$.
  Let $\alpha_0$ be the nontrivial geometric automorphism of $(\Qm_0,\rR_0)$;
  it is unique by hypothesis. By uniqueness, $\alpha_0$ is defined over $k$, as
  are $\Qm_0$ and $\rR_0$. Choose $L$ and $\varphi$ as in
  Lemma~\ref{lem:QuadExt}. The divisor $\sS_0 = \varphi_* (\rR_0)$ is
  $L$-rational, but it is even $k$-rational since
  \begin{equation*}
    \sS_0^{\sigma} = (\varphi_* (\rR_0))^{\sigma} = \varphi_*^{\sigma}
    (\rR_0^{\sigma}) = (\varphi \alpha )_* (\rR_0) = \varphi_* ( \alpha_*
    (\rR_0) ) = \varphi_* (\rR_0) = \sS_0  .
  \end{equation*}
  Now the degree~$2$ cover of $\PP^1$ with branch locus $\sS_0$ is
  $K$-isomorphic to $\Xm$. So since $X$ is $K$-isomorphic to a degree~$2$ cover
  of $\PP^1$ branching over a $k$-rational divisor, it admits a hyperelliptic
  equation over $k$.
\end{proof}

%\begin{remark}
  %As we will see in Section~\ref{sec:descent}, it is possible to construct
  %hyperelliptic curves with reduced automorphism group $\CG_2$ which do not
  %descend hyperelliptically, and which therefore do not descend at all.
%\end{remark}

\begin{proof}[Proof of Theorem~\ref{thm:MainTh1}]
  The case of even $g$ is due to Mestre in~\cite{mestre}, and Huggins proved the
  result in the case where $\Gbar$ is not tamely cyclic
  in~\cite[Thm.5.4]{huggins}. As for the case where $\Gbar$ is tamely cyclic of
  even order, this yields a pair $(\Qm_0,\rR_0)$ as in the proof of
  Proposition~\ref{prop:order2} whose reduced automorphism group is cyclic of
  even cardinality. Such a subgroup has a unique element $\alpha_0$ of order
  $2$, which is then defined over $k$ by uniqueness. It now suffices to invoke
  Proposition~\ref{prop:order2}.
\end{proof}

\subsection{An arithmetic criterion for hyperelliptic
descent}\label{sec:hypdescrit}

We can now characterize arithmetically whether a hyperelliptic curve $\Xm$
allows a hyperelliptic descent. Denote the quotient $\Xm / G$ by $\Bm$.  By
construction, $\Bm$ has a canonical Weil descent datum. Let $\Bm_0$ be the
corresponding model over $k$; its $k$-isomorphism class depends only on the
$K$-isomorphism class of $\Xm$. It is well-known (\emph{cf.}\ the discussion
in~\cite[Cor.2.3]{debes-emsalem}) that the existence of a $k$-rational point
on $\Bm_0$ implies that $\Xm$ descends.

%
%\begin{proposition} The $k$-isomorphism class of $\Bm_0$ depends only on the
%$K$-isomorphism class of $\Xm$.  \end{proposition} \begin{proof} Suppose we
%have two such curves $\Xm$ and $\Xm'$, with an isomorphism $i$ between them.
%Let $f_{\sigma}$ and $f'_{\sigma}$ be isomorphisms $\Xm \to
%\Xm^{\sigma}$ and $\Xm' \to \Xm'^{\sigma}$. These descend to cocycles
%$h_{\sigma} : \Bm \to \Bm^{\sigma}$ and $h'_{\sigma} : \Bm' \to
%\Bm'^{\sigma}$. These in turn are of the form $h^{\sigma} = d^{\sigma} d^{-1}$
%and $h'^{\sigma} = d'^{\sigma} d'^{-1}$ for some $d : \Bm_0 \to \Bm$ and
%$d' : \Bm'_0 \to \Bm'$.  The isomorphism $i$ descends to an isomorphism $g
%: \Bm \to \Bm'$.  For all $\sigma \in \Gamma$, we have $i^{\sigma}
%f^{\sigma} = f'^{\sigma} i$ up to an automorphism of $\Xm$, hence $g^{\sigma}
%h^{\sigma} = h'^{\sigma} g$.  Using this, one verifies that $d'^{-1} g d$ is
%fixed by Galois: \begin{equation*} (d'^{-1} g d)^{\sigma} = d'^{-\sigma}
%g^{\sigma} d^{\sigma} = d'^{-1} h'^{-\sigma} g^{\sigma} h^{\sigma} d = d'^{-1}
%g d  .  \end{equation*} Therefore there indeed exists an isomorphism $\Bm_0
%\to \Bm'_0$ over $k$.  \end{proof}
%
\begin{theorem}\label{thm:MainTh2}
  Let $\Xm$ be a hyperelliptic curve over $K$ of genus $g$ whose field of moduli
  for the extension $K | k$ equals $k$. Let $\Gbar$ be the reduced automorphism
  group of $\Xm$. Then $\Xm$ descends hyperelliptically if and only if the
  canonical model $\Bm_0$ of the quotient $\Bm = \Xm / G$ has a $k$-rational
  point.
\end{theorem}
\begin{proof} If $\Xm$ admits a hyperelliptic descent $\Xm_0$, then $\Bm_0$ has
  a rational point.  Indeed, the curve $\Bm_0$ can then be obtained as the
  quotient of $Q_0 = \Xm_0 / \iota_0 \cong \PP^1$ by the reduced automorphism
  group $\Gbar_0$ of $X_0$. Note that $\Gbar_0$ is defined over $k$, though its
  individual elements might not be.

  Conversely, if $\Bm_0$ has a $k$-rational point, then a descent $\Xm_0$ of
  $\Xm$ exists by~\cite{debes-emsalem}. In light of Theorem~\ref{thm:MainTh1},
  it then only remains to consider the case where the reduced automorphism group
  of $\Xm_0$ is tamely cyclic of odd order. So again let $\Qm_0$ be the quotient
  of $\Xm_0$ by its hyperelliptic involution. We get a map $\Qm_0 \to \Bm_0$ of
  odd degree, so that Lemma~\ref{lem:xarles}(ii) allows us to conclude that
  $\Qm_0$ is isomorphic with $\PP^1$ over $k$ as well. Therefore $\Xm_0$ is a
  degree $2$ cover of $\PP^1$ over $k$, so that $\Xm$ indeed descends
  hyperelliptically.
%  We get a morphism of odd degree
%  $\Qm_0 \to \Bm_0$ over $k$.  Pulling back the $k$-rational point on $\Bm_0$
%  under this morphism yields a $k$-rational divisor $\dD$ of odd degree, $2 n +
%  1$ say, on the conic $\Qm_0$.  This implies that $\Qm_0$ is isomorphic with
%  $\PP^1$ over $k$; if we let $\kK$ be the canonical divisor of $\Qm_0$, then
%  the $k$-rational divisor $\dD - n \kK$ is linearly equivalent to a
%  $k$-rational point on $\Qm_0$ by Riemann-Roch.
\end{proof}

The next proposition gives a concrete criterion for the presence of a rational
point on $\Bm_0$, which we will use in Section~\ref{sec:descent}. As usual, we
define the \emph{twist} of a curve $\Cm$ over $k$ to be a curve $\Cm'$ over $k$
that is isomorphic with $\Cm$ over $K$; we refer to~\cite[Ch.III.1]{serre-coho}
for a correspondence between the set of isomorphism classes of twists of $\Cm$
and the Galois cohomology set $H^1 (\Gal (K | k) , \Aut_K (\Cm))$.

\begin{proposition}\label{prop:WeilTest}
  Let $L$ be a quadratic extension of $k$, and let $\sigma$ be the nontrivial
  element of $\Gal (L | k)$. Let $\alpha_0 \in \Aut_k (\PP^1)$ be a
  $k$-automorphism of $\PP^1$ of order two defined over $k$, represented by an
  element $M_0$ of $\GL_2 (k)$. Let $c_K$ be the element of $H^1 (\Gal (K | k) ,
  \Aut_K (\PP^1))$ obtained by inflating the cocycle $c_L \in H^1 (\Gal (L | k)
  , \Aut_L (\PP^1)) = \Hom (\Gal (L | k) , \Aut_L (\PP^1))$ that sends $\sigma$
  to $\alpha_0$. Then the twist of $\PP^1$ over $k$ determined by $c_K$ is
  isomorphic to $\PP^1$ over $k$ if and only if $- \det (M_0)$ is a norm for the
  extension $L | k$.
\end{proposition}
\begin{proof}
  %After a suitable choice of homogeneous coordinates on $\PP^1_k$, we may
  %assume that $M$ is of the form $\left(\begin{array}{cc} 0 & 1 \\ \nu & 0
  %\end{array} \right)$. Indeed, starting with general $M$, we can change to a
  %basis of the form $\{ v , M v \}$. So we have to see if there exists an
  %invertible matrix $N$ over $L$ such that 
  Since the characteristic polynomial of $M_0$ is $x^2-\nu_0$ for some $\nu \in
  k$, its Frobenius companion matrix equals $\smat{0}{\nu_0}{1}{0 }$. The twist
  corresponding to $c_K$ is isomorphic to $\PP^1$ over $k$ if and only if $c_K$
  is a coboundary. This is the case if and only if there exists an invertible
  matrix $N$ over $L$ such that we have the equality $N^{\sigma} M_0 = N$ in
  $\PGL_2 (L)$, or more explicitly, if there exists some scalar $\lambda \in L$
  such that
  \begin{equation}\label{meq}
    N^{\sigma} = \lambda N \smat{0}{\nu_0}{1}{0 }^{-1}  .
  \end{equation}
  Writing out \eqref{meq} and eliminating, we get that $\lambda \in L$ and
  $\lambda^{\sigma} \lambda = \nu_0$, which shows that our condition is
  necessary.  Conversely, if such a $\lambda$ exists, then we can take
  \begin{displaymath}
    N = \smat{1}{\lambda^{\sigma}}{\beta}{\lambda^{\sigma} \beta^{\sigma}},
  \end{displaymath}
  where $\beta$ is any generator of $L$ over $k$.
\end{proof}

\section{Invariants}\label{sec:dihedral}

Let $\Xm$ be a hyperelliptic curve of genus $g$ over $K$, defined by a
homogeneous binary form $f$ over $K$ of degree $2g + 2$. Suppose that the
reduced automorphism group $\Gbar$ of $\Xm$ over $K$ is tamely cyclic of order
$n > 1$. In this section we will construct invariants of $f$ that can be used to
determine the descent obstruction (general or hyperelliptic) for $\Xm$, as well
as a corresponding descent of $\Xm$ if one of these obstructions vanishes. To
this end, we first construct geometric normal forms for $f$. Modulo a
normalization that we do not make, the discussion at the beginning of this
section is completely analogous to that in~\cite[Sec.2]{gutsha}.

Since we assumed that the reduced automorphism group of $\Xm$ is tame, we can
diagonalize one of its generators $\alpha$ over $K$. Making the corresponding
change of basis if necessary, we may therefore suppose that, using the notation
in the introduction,
\begin{equation}\label{eq:Gbar0}
  \Gbar = \CG_n = \left\langle \alpha \right\rangle  .
\end{equation}
The elements of $\Gbar$ then only have fixed points at $(0:1)$ and $(1:0)$.
Since we know that the binary form $f$ defining $\Xm$ is of even degree without
repeated roots, this implies that $f$ has one of the \emph{normal forms} over
$K$ figuring in the following definition.

\begin{definition}\label{def:typ}
  Let $n,m$ be positive integers. A binary form $f$ of even degree is said to
  be \emph{of type $(0,n,m)$}, resp.~$(1,n,m)$, resp.~$(2,n,m)$, if
  it is of the form
  \begin{align}
    f & = a_m x^{mn} + a_{m-1} x^{(m-1)n} z^n + \ldots + a_1 x^n z^{(m-1)n} + a_0
    z^{mn}  \label{ccform}  ,  && \text{resp.}  \\
    f & = z ( a_m x^{mn} + a_{m-1} x^{(m-1)n} z^n + \ldots + a_1 x^n z^{(m-1)n} +
    a_0 z^{mn})  \label{cform1}  , && \text{resp.} \\
    f & = x z (a_m x^{mn} + a_{m-1} x^{(m-1)n} z^n + \ldots + a_1 x^n z^{(m-1)n} +
    a_0 z^{mn}) . && \label{cform2}
  \end{align}
  for some $m$ and $n$, while also satisfying the following properties:
  \begin{enumerate}[(i)]
    \item $\Aut (f)$ coincides with the group $\CG_n$ from~\eqref{eq:Gbar0},
    and
    \item $f$ has no repeated linear factors.
  \end{enumerate}
\end{definition}

\begin{remark}
  We impose condition (ii) in Definition \ref{def:typ} to ensure that the forms
  $f$ under consideration define non-singular hyperelliptic curves.
  Condition (i) will hold for generic forms $f$ as in
  \eqref{ccform}-\eqref{cform2}. It is important that we restrict ourselves to
  this generic case by imposing (i), since the statement of our Theorem
  \ref{thm:MainTh3} essentially depends on the value of $m$, which in turn
  depends on that of $n$ once the degree of $f$ is fixed.
\end{remark}

As in~\cite{gutsha} or~\cite[Stz.5.2]{brandt-stich}, a calculation shows the
following.
\begin{proposition}
  The automorphism groups $\Aut (\Xm)$ of the hyperelliptic curves $\Xm : y^2 =
  f(x,z)$ defined by the forms in Definition~\ref{def:typ} are as follows.
  \begin{enumerate}[(i)]
    \item If $f$ is of type $(0,n,m)$, then $\Aut (\Xm)$ is isomorphic to the
      group $\CG_2 \times \CG_n$, generated by $(x:z:y) \mapsto (\zeta_n x : z :
      y)$ and $(x:z:y) \mapsto (x:z:-y)$. 
    \item If $f$ is of type $(1,n,m)$, then $\Aut (\Xm)$ is isomorphic to the
      group $\CG_{2 n}$, generated by $(x:z:y) \mapsto (\zeta_n x : z : -y)$.
    \item If $f$ is of type $(2,n,m)$, then $\Aut (\Xm)$ is isomorphic to the
      group $\CG_{2 n}$, generated by $(x:z:y) \mapsto (\zeta_n x : z :
      \zeta_{2n} y)$.
  \end{enumerate}
\end{proposition}

%A calculation shows that the full automorphism group of a member of the family
%of hyperelliptic curves determined by~\eqref{ccform} contains the group $\CG_2
%\times \CG_n$, generated by $(x:z:y) \mapsto (\zeta_n x : z : y)$ and $(x:z:y)
%\mapsto (x:z:-y)$. The automorphism groups of the members of the
%family~\eqref{cform1} (resp.~\eqref{cform2} contain the group $\CG_{2 n}$,
%generated by the transformation $(x:z:y) \mapsto (\zeta_n x : z : -y)$
%(resp.~$(x:z:y) \mapsto (\zeta_n x : z : \zeta_{2n} y)$).

Our methods now diverge from those of~\cite{gutsha}; we do not further normalize
to suppose $a_m = a_0 = 1$ so as to avoid breaking symmetry. This will make it
easy to transform $f$ to a normal form over an at worst quadratic extension of
the base field $k$, as we shall see Proposition~\ref{prop:f0}.
%This condition affects the field of definition of the standard forms. It is the
%determination of the invariants after this demand has been imposed that causes
%the dihedral invariants in \emph{loc. cit.} to define an extension of $k$ in
%general, as mentioned in the introduction.

\subsection{Restricting isomorphisms}

%Let
%\begin{equation}\label{eq:T}
%  T =  \left\lbrace \smat{\lambda}{0}{0}{\mu}, \smat{0}{1}{1}{0} \; : \;
%  \lambda , \mu \in  K  \right\rbrace \subset \GL_2 (K)
%\end{equation}
%be the subset of diagonal matrices. Define
%
%First we define the binary forms that we will consider in the following three
%sections, namely those corresponding to a certain class of hyperelliptic curves
%with prescribed automorphism group. The terminology in the following definition
%will be explained and generalized in Section~\ref{sec:other}.
%
%\begin{definition}\label{def:typ0}
%  Let $n,m$ be positive integers such that $m = 2 \ell$ is even. A
%  binary form $f$ is said to be of \emph{type $(0,n,m)$} if it is of the
%  form~\eqref{ccform} for some $m$ and $n$, while also satisfying the following
%  properties:
%  \begin{enumerate}[(i)]
%    \item $\Aut (f)$ coincides with the group $\CG_n$ from~\eqref{eq:Gbar0},
%    and
%    \item $f$ has no repeated roots.
%  \end{enumerate}
%\end{definition}

We start by determining the possible isomorphisms between two binary forms of
type $(i,n,m)$. Let $T \subset \GL_2 (K)$ be the subgroup of diagonal matrices
and define
\begin{equation}\label{eq:D}
  D = \left\langle T, \smat{0}{1}{1}{0} \right\rangle  ,
\end{equation}
which is an extension of $\ZZ / 2 \ZZ$ by $T$.

\begin{proposition}\label{prop:NormD}
  Consider two binary forms $f , f'$ of type $(i,n,m)$ with $n > 1$. Suppose
  that $A \in \GL_2 (K)$ is such that $f' \sim A . f$. Then $A \in D$, and
  moreover $A \in T$ if $i = 1$.
\end{proposition}
\begin{proof}
  Under the hypotheses of the Proposition, we have that $A g A^{-1} f' \sim A g
  A^{-1} A f \sim A g f \sim A f \sim f'$ for all $g \in \Aut (f)$, or in other
  words $A \Aut (f) A^{-1} \subset \Aut
  (f')$. Since hypothesis (i) in Definition~\ref{def:typ} is verified for both
  $f$ and $f'$, we therefore see that $A \CG_n A^{-1} = \CG_n$, showing that $A$
  is indeed in the normalizer of $\CG_n$.  The inclusion $A \in D$ then results
  from the description of this normalizer in~\cite[Lem.3.3]{huggins}.

  In the case $i = 1$ we can conclude that $A \in T$ because the matrix
  $\smat{0}{1}{1}{0}$ sends $z$ to $x$, which is impossible since $f$ being
  of type $(1,n,m)$ implies that the coefficients in~\eqref{cform1} satisfy $a_m
  a_0 \neq 0$.
\end{proof}

In the following exposition, we will first focus on the normal
form~\eqref{ccform} with $m = 2 \ell$ even. The other cases are discussed in the
final Section~\ref{sec:other}.

\subsection{Homogeneous diagonal invariants}

We want to develop the invariant theory of binary forms of type $(0,n,2\ell)$ under
the action of the group $D$. We first consider the action of the simpler index
$2$ subgroup $T$ of $D$. On the coefficients in~\eqref{ccform}, the action of an
element $\smat{\lambda}{0}{0}{\mu}$ of $T$ is given by
\begin{small}
  \begin{equation}\label{cctrans}
    (a_{m} , a_{m-1} , \ldots , a_1 , a_0) \longmapsto (\lambda^{m n} a_{m} ,
    \lambda^{(m - 1) n} \mu^n a_{m-1} , \ldots , \lambda^n \mu^{(m - 1) n} a_1
    , \mu^{m n} a_0 ) .
  \end{equation}
\end{small}
We now wish to consider the homogeneous invariants under this action, that is,
those polynomial expressions that are actually invariant under the action of the
proper subgroup $T \cap \SL_2 (K)$ of $T$. The ring of these invariants admits a
weight decomposition under the action of the full group $T$; an element $I$ is of
weight $w$ if $A \in T$ sends $I$ to $\det(I)^{w \ell} I$. More intuitively,
this simply means that $I$ has degree $w$ as a homogeneous polynomial.

We will construct small systems of invariants that allow us to distinguish the
orbits of binary forms of type $(0,n,m)$ under the action of $T$. First we
consider the following homogeneous invariants for the action of $T$, which will
turn out to suffice for distinguishing most of these binary forms:
\begin{small}
  \begin{displaymath}
    \begin{array}{lrcll}
      \mathrm{Degree}\ 1: & J_1 & = & a_{\ell}, &\\[0.15cm]
      \mathrm{Degree}\ 2: &J_{2,0} & = & a_{2 \ell} a_0 ,\ &%  \\
      J_{2,1}  =  a_{2 \ell -1} a_1 ,\  %\\
      \ldots,\  %\\
      J_{2,\ell-1}  =  a_{\ell+1} a_{\ell-1}, \\[0.15cm]
      \mathrm{Degree}\ 3: &J_3 & = & a_{\ell+2} a_{\ell-1}^2  &\\[0.15cm]
      \mathrm{Degree}\ 4: &J_4 & = & a_{\ell+3} a_{\ell-1}^3  &\\[0.15cm]
      && \vdots & &\\
      \mathrm{Degree}\ \ell+1: &J_{\ell + 1} & = & a_{2\ell} a_{\ell-1}^{\ell}.&
    \end{array}  
  \end{displaymath}
\end{small}
The first index for these invariants indicates their homogeneous degree.

\begin{definition}\label{def:gendiaginv}
  We call the invariants $J_1 , J_{2,0} , \ldots , J_{2,\ell-1} , J_3 ,
  \ldots J_{\ell+1}$ defined above the \emph{generic homogeneous diagonal
  invariants} (for binary forms of type $(0,n,m)$).
\end{definition}

\begin{example}\label{ex:0n4diag}
  For forms $f$ of type $(0,n,4)$ given by $f = a_4 x^{4 n} + a_3 x^{3 n} z^{n}
  + a_2 x^{2 n} z^{2 n} + a_1 x^n z^{3 n} + a_0 z^{4 n}$, the generic
  homogeneous diagonal invariants are given by $J_1 = a_2 , J_{2,0} = a_4 a_0 ,
  J_{2,1} = a_3 a_1$ and $J_3 = a_4 a_1^2$. Note that the case $n = 2$ yields a
  class of hyperelliptic genus $3$ curves with extra involutions.
\end{example}

Using the generic homogeneous dihedral invariants already suffices to deal with
most binary forms of type $(0,m,n)$:

\begin{proposition}\label{prop:GenTInv}
  Suppose that $f$ and $f'$ are binary forms of type $(0,n,m)$ such that
  \begin{equation*}
    a_{2\ell} , a_{2\ell-1} , \ldots a_{\ell+2}, a_{\ell-1} \neq 0
  \end{equation*}
  and
  \begin{equation*}
    a'_{2\ell} , a'_{2\ell-1} , \ldots a'_{\ell+2}, a'_{\ell-1} \neq 0 .
  \end{equation*}
  If the generic homogeneous diagonal invariants $J$ and $J'$ of $f$ and $f'$
  define the same point in the corresponding weighted projective space, then
  there exists an $A \in T$ such that $f' \sim A  . f$.
\end{proposition}
\begin{proof}
  Scaling if necessary, we may suppose that $J$ and $J'$ are equal. Then a
  suitable modification by a matrix of the form
  $\smat{\lambda}{0}{0}{\lambda^{-1}}$ can be used to ensure that $a_{\ell - 1}
  = a'_{\ell - 1}$. This does not affect the equality of $J$ and $J'$ since this
  matrix has trivial determinant. Our result is now clear, since the other $a_i$
  can be read off from the values of the generic homogeneous diagonal invariants
  once $a_{\ell - 1} \neq 0$ is known.
\end{proof}

%\begin{remark}
%  Proposition~\ref{prop:GenTInv} does not exclude that other transformations may
%  transform $f'$ and $f$ into one another. For example, consider the case where
%  $f = a_2 x^4 + a_1 x^2 z^2 + a_0 z^4$ and $f' = a_0 x^4 + a_1 x^2 z^2 +a_2
%  z^4$. Then both of the transformations $(x : z) \mapsto (z : x)$ and $(x : z)
%  \mapsto (r x : r^{-1} z)$ (with $r = \sqrt[4]{a_0 / a_2}$) exchange $f$ and
%  $f'$.
%\end{remark}

In the non-generic case (\ie, when one of the conditions of
Proposition~\ref{prop:GenTInv} is not satisfied), the construction of the
appropriate homogeneous diagonal invariants is slightly more complicated. To
proceed in these cases, we first note that the set of indices of the
coefficients $a_j$ of $f$ that are nonzero do not change under the action of
$T$, and also that $a_0$ and $a_{2\ell}$ are never zero. Considering these
indices allows us to determine which small set of modified invariants we need to
use.

\begin{definition}\label{def:Sadm}
  Let $f$ be a binary form of type $(0,n,m)$ as in~\eqref{ccform}. Given an
  integer $r \leq m + 1$ and a tuple $S = (s_1 , \ldots , s_r )$ of distinct
  integers in $\{ 0, \dots ,m \}$, we say that $f$ is \emph{$S$-admissible} if
  \begin{itemize}
    \item[(S1)] $a_s \neq 0$ for all $s \in S$ and
    \item[(S2)] if one of $a_i$, $a_{2 \ell - i}$ is nonzero, then exactly one
      element of $\{ i , 2 \ell - i \}$ is in $S$.
  \end{itemize}
\end{definition}

%Suppose that we know this set of vanishing indices. Let $I$ be the sequence
%\begin{equation}
  %I = ( i \in \{0 , \ldots , \ell - 1 \} : a_i = a_{2 \ell - i} = 0 ) ,
%\end{equation}
%and define the sequence %\begin{equation}
%$
%S = ( s_i \; | \; i \in \{0 , \ldots , \ell - 1 \} \backslash I ) ,
%$
%%\end{equation}
%where
%\begin{equation}
  %s_i = \left\lbrace \begin{array}{cl} i & \textrm{if \;} a_{2 \ell - i} \neq 0
  %\textrm{\; and \;}  i \neq 0 , \\ 2 \ell - i & \textrm{otherwise.}
  %\end{array} \right.
%\end{equation}
%Using the same $J_1$ and $J_{2,j}$ as above,

Clearly every binary form $f$ of type $(0,n,m)$ is $S$-admissible for some
$S$. We now construct the homogeneous invariants of $T$ that are monomials in
the $\{ a_s : s \in S \}$.

\begin{proposition}\label{prop:NonGenTInv}
  Under the hypotheses (S1)-(S2), associate with $S$ the single-row matrix $M_S
  = (s_1 - \ell , \ldots , s_{r} - \ell)$ over $\QQ$. Then the elements of $\ker
  (M_S) \cap \NN^{r}$ are in one-to-one correspondence with the homogeneous
  invariants of $T$ for the family of $S$-admissible binary forms that are
  monomials in $\{ a_s : s \in S \}$, by the association $v \leftrightarrow
  \prod_{i = 1}^{r} a_{s_i}^{v_i}$.
\end{proposition}
\begin{proof}
  This follows from the transformation behavior of the exponents of the
  coefficients $a_i$, which is given in~\eqref{cctrans}.
\end{proof}

Generalizing Proposition~\ref{prop:GenTInv}, it turns out that together with the
invariants $J_1 , J_{2,0} , \ldots J_{2,\ell-1}$ these new homogeneous diagonal
invariants allow one to reconstruct an $S$-admissible binary form $f$, as the
following proposition shows.

\begin{proposition}\label{prop:param}
  Let $f$ and $f'$ be two $S$-admissible binary forms of type $(0,n,m)$. There
  exists a finite subset $R$ of the invariants constructed in
  Proposition~\ref{prop:NonGenTInv} with the property that there exists an $A
  \in T$ such that $f' \sim A  . f$ if and only if the values of the invariants
  of $f$ and $f'$ at $R \cup \{ J_{2,0} , \ldots J_{2,\ell - 1} \}$ determine
  the same point in the corresponding weighted projective space.
\end{proposition}
\begin{proof}
  We will assume that $\#S>1$, since the case $S = 1$ is easy. Before starting
  our construction, we modify $S$; if $M_S$ consists completely of either only
  strictly positive or only strictly negative elements, then we change the entry
  $2 \ell$ of $S$ to $0$ or inversely. Note that $f$ and $f'$ will still be
  $S$-admissible after this change since $a_0 a_{2 \ell} \ne 0$.

  We first construct a $\ZZ$-basis of the $\ZZ$-module $K_S = \ker (M_S) \cap
  \ZZ^{S}$ used in Proposition~\ref{prop:NonGenTInv}. The module $K_S$ is
  torsion-free, since it is a submodule of a torsion-free $\ZZ$-module.
  Furthermore, the quotient $\ZZ^{S} / K_S$ is torsion-free as well. Indeed,
  suppose that $n x \in K_S$ for some $n \in \ZZ$ and $x \in \ZZ^{S}$. Then $M_S
  (n x) = 0$, so $M_S (x) = 0$ and $x \in \ker (M_S) \cap \ZZ^{S} = K_S$. We
  thus have an exact sequence of finitely generated free $\ZZ$-modules
  \begin{equation}\label{eq:splitseq}
    0 \longrightarrow K_S \longrightarrow \ZZ^{S} \longrightarrow \ZZ^{S} / K_S
    \longrightarrow 0 .
  \end{equation}
  Choose a basis $\{v_i \}_{i = 1}^{\#S - 1}$ of $K_S$. Then since sequence
  \eqref{eq:splitseq} is split, there exists a vector $w \in \ZZ^{S}$ such that
  $\ZZ^{S}$ has basis $\{ v_i \}_{i = 1}^{\#S - 1} \cup \{ w \}$.

  We will now construct an element $v \in K_S$ such that all the entries of $v$
  are strictly positive. To accomplish this, we note that not all the entries of
  $M_S$ have the same sign, since in this case $f$ would have repeated roots.
  Therefore, given an index $i$ of $M_S$, we can find another
  index $j$ such that $(M_S)_i$ and $(M_S)_j$ have the opposite sign. We can now
  construct an elements of $K_S$ whose only nontrivial entries are at $i$ and
  $j$, with values $(M_S)_j$ and $-(M_S)_i$. Multiplying by $-1$ if necessary,
  we get an element of $\NN^S \cap K_S$ that is nontrivial at the index $i$.
  Summing over the indices $i$ now gives the requested element $v$ of $K_S$.

  We now claim that there exists a basis $\{v_i \}_{i = 1}^{\#S - 1} \cup \{
  w \}$ of $\ZZ^S$ all of whose elements are in $\NN^{S}$. To see this, consider
  the element $v$ constructed in the previous paragraph and choose $v_1$ such
  that $\NN v_1 = \QQ v \cap \NN^S$.  Then $v_1$ also has all of its entries
  strictly positive. Moreover, $v_1$ can again be completed to a basis $\{ v_i
  \}_{i = 1}^{\#S - 1}$ of $K_S$ because the same argument used to produce the
  sequence~\eqref{eq:splitseq} shows the existence of a split exact sequence
  of free $\ZZ$-modules
  \begin{equation*}
    0 \longrightarrow \ZZ v_1 \longrightarrow K_S \longrightarrow K_S  / \ZZ v_1
    \longrightarrow 0 .
  \end{equation*}
  It then only remain to add sufficiently large multiples of $v_1$ to the other
  elements of the resulting basis. This yields the requested basis $\{v_i \}_{i
  = 1}^{\#S - 1}$ of $K_S$, which we can augment to a basis $\{ v_i \}_{i =
  1}^{\#S - 1} \cup \{ w \}$ of $\ZZ^S$ as before. Moreover, by adding multiples
  of the $v_i$ to $w$, we can insure that $w$ is in $\NN^{S}$ as well, which
  implies our claim.

  After these preparations, the proof of the proposition is straightforward. The
  monomials corresponding to the basis elements $v_i$ under the correspondence
  in Proposition~\ref{prop:NonGenTInv} will now play the role of the generic
  homogeneous diagonal invariants; they will turn out to distinguish the orbits
  under $T$ of $S$-admissible binary forms. We let $t = \prod_{i = 1}^{\#S}
  a_{S_i}^{w_i}$ be the monomial corresponding to $w$. Since $w$ is not in
  $K_S$, we can use matrices of the form $\smat{\lambda}{0}{0}{\lambda^{-1}}$ as
  in the proof of Proposition~\ref{prop:GenTInv} to suppose that the value of
  $t$ is the same for $f$ and $f'$ without affecting the value of the invariants
  corresponding to the $v_i$. And as in that same proof, knowing $t$ and the
  value of these invariants along with the invariants $J_1 , J_{2,0} , \ldots
  J_{2,\ell - 1}$ determines the coefficients of the binary forms involved.
  Indeed, because $\ZZ^S$ has basis $\{ v_i \}_{i = 1}^{\#S - 1} \cup \{ w \}$,
  we can reconstruct the nonzero coefficients $\{a_s : s \in S \}$. The
  invariants $J_{2,0} , \ldots J_{2,\ell - 1}$ then determine the other
  coefficients by property (S2) in Definition~\ref{def:Sadm}.
\end{proof}

\begin{definition}\label{def:Sdiaginv}
  Given a binary form $f$ of type $(0,n,m)$ and any tuple $S$ for which $f$ is
  $S$-admissible, we call any of the finite sets $R \cup \{ J_{2,0} , \ldots J_{2,\ell -
  1} \}$ constructed in Proposition~\ref{prop:NonGenTInv} the
  \emph{homogeneous diagonal invariants} of $f$.
%  Let $S$ be a set of coefficient indices, as in Definition~\ref{def:Sadm}. We
%  call the invariants constructed in Proposition~\ref{prop:NonGenTInv} the
%  \emph{$S$-homogeneous diagonal invariants} (for binary forms of type
%  $(0,n,m)$). By abuse of notions (but in line with
%  Definition~\ref{def:gendiaginv}, see the upcoming Example~\ref{ex:TEx1}), we
%  also call a finite set $R \cup \{ J_{2,0} , \ldots J_{2,\ell - 1} \}$ as in
%  Proposition~\ref{prop:param} the \emph{$S$-homogeneous diagonal invariants}
%  (for binary forms of type $(0,n,m)$).
%
%  By abuse of notion, given a binary form $f$ of type $(0,n,m)$, we will call
%  any set of $S$-homogenous diagonal invariants for some $S$ for which $f$ is
%  $S$-admissible the \emph{homogeneous diagonal invariants} of $f$.
\end{definition}

\begin{remark}
  It may seem unnatural to modify invariants depending on the vanishing behavior
  of the coefficients of $f$, but in practice this is very useful, since the
  parametrization from Corollary~\ref{cor:param} is crucial for our
  reconstruction purposes. We again emphasize that once an initial binary form
  $f$ of type $(0,n,m)$ is given, one sees immediately which invariants should
  be used; indeed, the set $S$ that one can take in Definition~\ref{def:Sadm}
  are purely determined by the vanishing behavior of the coefficients of $f$.
\end{remark}

\begin{corollary}\label{cor:param}
  The set of $S$-admissible binary forms with given $S$-homogeneous diagonal
  invariants is a rational space of dimension $1$.
\end{corollary}
\begin{proof}
  This is clear from the proof of Proposition~\ref{prop:param}. Indeed, the
  given set is parametrized by the monomial corresponding to the complementary
  vector $w$.
\end{proof}

\begin{example}\label{ex:TEx1}
  The generic homogeneous diagonal invariants from
  Definition~\ref{def:gendiaginv} correspond to the case where $S = (2 \ell , 2
  \ell - 1 , \ldots , \ell + 2 , \ell - 1)$, so $M_S = (\ell , \ell - 1 , \ldots
  , 2 , -1)$. The resulting kernel $K_S$ has an ordered basis consisting of the
  positive elements
  \begin{equation*}
    \begin{array}{c}
      ( 1 , 0 , 0 , \ldots , 0, 0 , \ell) , \\
      ( 0 , 1 , 0 , \ldots , 0, 0 , \ell - 1) , \\
      \ldots \\
      ( 0 , 0 , 0 , \ldots , 0, 1 , 2) . \\
    \end{array}
  \end{equation*}
  corresponding to the generic invariants $J_{\ell + 1} , J_{\ell} , \ldots ,
  J_{3}$, respectively. The complementary element $(0 , 0 , \ldots , 0 , 1)$
  corresponds to $a_{\ell - 1}$, which can indeed be used to parametrize the
  corresponding rational spaces, as we have seen in the proof of
  Proposition~\ref{prop:GenTInv}.
\end{example}

\begin{example}\label{ex:TEx2}
  Let $\ell = 6$ and take $S = (12,8,3,1)$. A basis for $K_S$ in $\NN^4$ is
  given by $\{ (3,0,1,3),(3,1,0,4),(5,0,0,6) \}$, and a complementary element
  $w$ is furnished by $(1,0,0,1)$. This shows that for binary forms
  $f$ of type $(0,n,m)$ such that $$a_{10} = a_7 = a_5 = a_2 = 0$$ and $a_{8}, a_{3} , a_1 \neq
  0$, a set of $S$-homogeneous diagonal invariants is furnished by
  \begin{eqnarray*}
    J_1 &=& a_6 , \\
    J_{2,0} &=& a_{12} a_{0} , J_{2,1} = a_{11} a_{1} , J_{2,3} =
    a_9 a_3 , J_{2,4} = a_8 a_4 ,\\
    J_{7} &=& a_{12}^3 a_3 a_1^3 , \\
    J_{8} &=& a_{12}^3 a_8 a_1^4,
    \\J_{11} &=& a_{12}^5 a_1^6  .
  \end{eqnarray*}
  Moreover, we can use $w = a_{12} a_{1}$ to parametrize the corresponding
  rational spaces of binary forms with given $S$-homogeneous invariants.
\end{example}

\begin{remark}
  A uniform approach to the problem is also available, namely by constructing
  the full invariant algebra of the action of $T$ on the general binary
  form~\eqref{ccform} in Definition~\ref{def:typ}. This can be done by writing
  down the invariant monomials of given weight, adding the result to the set of
  generator if it is not an expression in the monomials already found. By a
  result of Wehlau \cite{wehlau}, this process always terminates at degree
  $m-1$.  A script to generate this invariant algebra is available
  online\footnoteref{foothyp}. For the case $m = 8$, it is generated by the
  expressions
  \begin{small}
    \begin{align*}
      \mathrm{Degree}\ 1:\ && a_4, \\[0.15cm]
      \mathrm{Degree}\ 2:\ && a_7 a_1,\ &%\\
      a_6 a_2,\ &%\\
      a_5 a_3,\ &%\\
      a_8 a_0,  \\[0.15cm]
      \mathrm{Degree}\ 3:\ &&a_8 a_3 a_1,\ &a_7 a_5 a_0,& %\\
      a_7 a_3 a_2,\ &a_6 a_5 a_1,\ \\
      &&a_8 a_2^2,\ &a_6^2 a_0,\ &%\\
      a_6 a_3^2,\  &a_5^2 a_2, \\[0.15cm]
      \mathrm{Degree}\ 4:\ &&a_8 a_6 a_1^2,\ &a_7^2 a_2 a_0,\ &%\\
      a_8 a_5 a_2 a_1,\  &a_7 a_6 a_3 a_0,\   \\
      &&a_7 a_5 a_2^2,\ &a_6^2 a_3 a_1,\ &%\\
      a_7 a_3^3,\  &a_5^3 a_1,\   \\
      &&a_8 a_3^2 a_2,\  &a_6 a_5^2 a_0,   \\[0.15cm]
      \mathrm{Degree}\ 5:\ &&a_8^2 a_2 a_1^2,\ &a_7^2 a_6 a_0^2,\ %\\
      &a_8 a_5^2 a_1^2,\  &a_7^2 a_3^2 a_0,\   \\
      &&a_7^2 a_2^3,\ &a_6^3 a_1^2,\ %\\
      &a_8 a_3^4,\  &a_5^4 a_0,   \\[0.15cm]
      \mathrm{Degree}\ 6:\ &&a_8^2 a_5 a_1^3,\  &a_7^3 a_3 a_0^2,   \\[0.15cm]
      \mathrm{Degree}\ 7:\ &&a_8^3 a_1^4,\ &a_7^4 a_0^3\,.
    \end{align*}
  \end{small}
  The non-generic invariants constructed in Proposition~\ref{prop:NonGenTInv}
  are of course expressions in these monomials.  Theoretically, this uniform
  approach is much more satisfying, but the results get unwieldy for bigger $m$;
  the number of invariants runs into the hundreds for $m\geq 12$, whereas by
  contrast, the number of homogeneous invariants constructed above is always at
  most $\ell + 1$, no matter which subset $S$ of coefficients is considered.
\end{remark}

\subsection{Homogeneous dihedral invariants}

We resume the main thread of our argument. Now that we have determined useful
small sets of invariants for the action of the normal subgroup $T \subset D$ of
index~$2$, we can construct the invariants for $D$ itself by a symmetrization.
Before starting, we need an elementary result.

\begin{lemma}\label{lem:C2Quot}
  Let $n \geq 1$, and let $X$ be the affine space with coordinates $(s_1 ,
  \ldots , s_n , t_1 , \ldots t_n)$. Define an action of the cyclic group
  $\CG_2$ on $X$ by $s_i \leftrightarrow t_i$. Consider the invariants $\{s_i +
  t_i\}_{i=1}^n \cup \{s_i t_j + s_j t_i \}_{i,j=1}^n$ of this action. Then the
  orbit under the action of $\CG_2$ of a point $x \in X$ is determined by these
  invariants.
\end{lemma}
\begin{proof}
  Certainly the subset of invariants $\{s_i + t_i\}_{i=1}^n \cup \{2 s_i t_i \}_{i=1}^n$
  determines $x = (s_1 , \ldots , s_n , t_1 , \ldots t_n)$ up to some sequence
  of exchanges $s_i \leftrightarrow t_i$.
  We have to show that the additional invariants suffice to tell apart a
  sequence of such exchanges, except when either none or all of the $s_i$ and
  $t_i$ are exchanged. So suppose that we have two indices $i$ and $j$ where
  $s_i \neq t_i$ and $s_j \neq t_j$, and we exchange $s_i$ and $t_i$ while
  leaving the coordinates with index $j$ fixed. Then equality of the invariants
  yields $s_i t_j + s_j t_i = t_i t_j + s_j s_i$, hence $(s_i - t_i)(s_j - t_j)
  = 0$, a contradiction with our hypothesis.
\end{proof}

By using Lemma~\ref{lem:C2Quot}, we can now find small sets of homogeneous
invariants that can be used to distinguish orbits of binary forms $f$ of type
$(0,n,m)$.  First we consider the generic case. Let $J'_i$ denote the
transformation of the invariant $J_i$ under the involution $a_i \mapsto a_{m-i}$
on the coefficients, and let
\begin{small}
  \begin{align*}
    I_1 &= J_1,\\
    I_{2,0} &= J_{2,0},\ %\\
    I_{2,1} = J_{2,1},\ %\\
    \ldots,\ %\\
    &I_{2,\ell-1} &= J_{2,\ell-1},  \\
    I_{3,3,1} &= J_3 + J'_3,\ %\\
    &I_{3,3,2} &= J_3 J'_3,\  &\\
    & \vdots & \vdots \\
    I_{\ell + 1,\ell + 1,1} &= J_{\ell+1} + J'_{\ell+1},\ %\\
    &I_{\ell + 1,\ell + 1,2} &= J_{\ell+1} J'_{\ell+1},\  \\[0.15cm]
    I_{3,4} &= J_3 J'_4 + J'_3 J_4,\ %\\
    &I_{3,5} &= J_3 J'_5 + J'_3 J_5,\ %\\
    \ldots,\ %\\
    I_{3,\ell + 1} = J_3 J'_{\ell + 1} + J'_3 J_{\ell + 1}, \\[0.15cm]
    I_{4,5} &= J_4 J'_5 + J'_4 J_5,\ 
    &I_{4,6} &= J_4 J'_6 + J'_4 J_6,\ %\\
    \ldots,\ %\\
    I_{4,\ell + 1} = J_4 J'_{\ell + 1} + J'_4 J_{\ell + 1}, \\[0.15cm]
    &\vdots  \\
    I_{\ell,\ell+1} &= J_\ell J'_{\ell+1} + J'_\ell J_{\ell+1} .
  \end{align*}
\end{small}
These expressions are homogeneous invariants under the action of $D$.  Though
$D$ is not a dihedral group, we still employ the following terminology, which
was introduced in~\cite{gutsha}.

\begin{definition}\label{def:gendihinv}
  We call the symmetrized invariants $I_{\cdot}$ defined above the \emph{generic
  homogeneous dihedral invariants} (for binary forms of type $(0,n,m)$).
\end{definition}

\begin{example}\label{ex:0n4dihedral}
  For the forms $f$ of type $(0,n,4)$ considered in
  Example~\ref{ex:0n4dihedral}, the generic homogeneous diagonal invariants are
  given by $I_1 = J_1 = a_2 , I_{2,0} = J_{2,0} = a_4 a_0 , I_{2,1} = J_{2,1} =
  a_3 a_1, I_{3,3,1} = J_3 + J'_3 = a_4 a_1^2 + a_3^2 a_0$ and $I_3 = J_3 J'_3 =
  a_4 a_3^2 a_1^2 a_0$.
\end{example}

\begin{remark}\label{rem:smaller}
  As long as $J_3 \neq J'_3$ the invertible linear systems in $J_i$ and $J'_i$
  given by considering two of the invariants $I_{i,i,1} = J_i + J'_i$ and
  $I_{3,i} = J_3 J'_i + J'_3 J'_i$ are invertible. Therefore, we can usually
  even get by with a further subset of these generic homogeneous dihedral
  invariants in our calculations, namely $I_1 , I_{2,0} \ldots I_{2,\ell+1},
  I_{3,3,1} , \ldots I_{\ell+1,\ell+1,1}, I_{3,3,2}, I_{3,4} , \ldots
  I_{3,\ell+1}$. In Example~\ref{ex:nondesc1}, we take this approach.
\end{remark}
  
The symmetrization process is similarly straightforward for the $S$-homogeneous
diagonal invariants, so we can also construct homogeneous dihedral invariants in
the non-generic cases.

\begin{definition}\label{def:Sdihinv}
  Given a binary form $f$ of type $(0,n,m)$ with $m = 2 \ell$ even and any tuple
  $S$ for which $f$ is $S$-admissible, we call the symmetrization of any of the
  finite sets $R \cup \{ J_{2,0} , \ldots J_{2,\ell - 1} \}$ constructed in
  Proposition~\ref{prop:NonGenTInv} the \emph{homogeneous dihedral invariants}
  of $f$.
\end{definition}

\begin{proposition}\label{prop:GenDInv}
  Let $T = (0,n,m)$ be a type with $m = 2 \ell$ even.
  \begin{enumerate}
    \item Suppose that $f$ and $f'$ in~\eqref{ccform} of type $T$ are such that
      \begin{enumerate}[(i)]
        \item either $a_{2 \ell} , a_{2 \ell - 1} , \ldots a_{\ell+1}, a_{\ell-1}
          \neq 0$ or $a_{\ell+1} , a_{\ell-1} , \ldots a_{1}, a_{0} \neq 0$ and
        \item either $a'_{2 \ell} , a'_{2 \ell - 1} , \ldots a'_{\ell+1},
          a'_{\ell-1} \neq 0$ or $a'_{\ell+1} , a'_{\ell-1} , \ldots a'_{1}, a'_{0}
          \neq 0$.
      \end{enumerate}
      If the generic homogeneous dihedral invariants $I$ and $I'$ of $f$ and $f'$
      define the same point in the corresponding weighted projective space, then
      there exists an $A \in D$ such that $f'  \sim A . f$.
    \item For general $S$-admissible $f$ and $f'$ whose invariants define the
      same point in the corresponding weighted projective space, the same
      conclusion holds.
  \end{enumerate}
\end{proposition}
\begin{proof}
  Note that the conditions of part (i) of the proposition are
  indeed invariant under the action of $D$. Using Lemma~\ref{lem:C2Quot}, we see
  that replacing $f'$ by its transformation by $\smat{0}{1}{1}{0}$ if necessary,
  we may assume that $f$ and $f'$ have the same homogeneous diagonal invariants.
  Then the parametrization by $a_{\ell - 1}$ in Proposition~\ref{prop:GenTInv}
  allows us to conclude.

  For general forms, the argument is essentially the same, replacing the
  parametrizing element $a_{\ell - 1}$ by the monomial corresponding to $w$ in
  Proposition~\ref{prop:param}.  Note that forms in the same $D$-orbit are
  indeed $S$-admissible for the same $S$, so that the same set of homogeneous
  dihedral invariants can be used.
\end{proof}

The construction of general homogeneous dihedral invariants is perhaps best
illustrated by an example.

\begin{example}\label{ex:DEx}
  Consider the binary forms $f$ of type $(0,n,12)$ such that both $a_2
  = a_5 = a_7 = a_{10} = 0$ and either $a_{11} , a_9 , a_4 \neq 0$ or $a_8 , a_3
  , a_1 \neq 0$. The symmetrization of the invariants in Example~\ref{ex:DEx}
  yields the following $S$-homogeneous dihedral invariants for this family:
  \begin{small}
    \begin{align*}
      I_1  &=  a_6, \\
      I_{2,0} &= a_{12} a_0,\ I_{2,1} = a_{11} a_1,\ &I_{2,3} &= a_9 a_3,\
      I_{2,4}  =  a_8 a_4,\\
      I_{7,7,1} &= a_{11}^6 a_0^5 + a_{12}^5 a_1^6,\
      &I_{7,7,2}  &=  a_{12}^5 a_{11}^6 a_1^6 a_{12}^5,  \\
      I_{8,8,1} &= a_{11}^3 a_9 a_0^3 + a_{12}^3 a_3 a_1^3,\
      &I_{8,8,2}  &=  a_{12}^3 a_{11}^3 a_9 a_3 a_1^3 a_0^3,  \\
      I_{11,11,1} &= a_{11}^4 a_4 a_0^3 + a_{12}^3 a_8 a_1^4,\
      &I_{11,11,2}  &=  a_{12}^3 a_{11}^4 a_8 a_4 a_1^4 a_0^3,  \\
      I_{7,8} &= a_{12}^3 a_{11}^6 a_3 a_{1}^3 a_0^5 + a_{12}^5 a_{11}^3 a_9
      a_1^6a_0^3,\
      &I_{7,11}  &=  a_{12}^3 a_{11}^6 a_8 a_1^4 a_0^5 + a_{12}^5 a_{11}^4 a_4 a_1^6 a_0^3,\\
      I_{8,11} &= a_{12}^3 a_{11}^3 a_9 a_8 a_1^4 a_0^3 + a_{12}^3 a_{11}^4
      a_4 a_3 a_1^3 a_0^3.
    \end{align*}
  \end{small}
\end{example}

\begin{remark}\label{sec:arithm-dihedr-invar}
  The homogeneous dihedral invariants of a general binary octavic form $$f = a_8 x^{8} +
  a_{7} x^{7} z + \ldots + a_1 x z^{7} + a_0 z^{8}$$ are the following:
  \begin{displaymath}
  \begin{array}{lrl}
    \mathrm{Degree}\ 1:\ &i_1 &= { a_4},\medskip\\
    \mathrm{Degree}\ 2:\ &i_{2} &= { a_0}\,{ a_8},\ %
    j_{2} = { a_1}\,{ a_7},\ %
    k_{2} = { a_2}\,{ a_6},\ %
    l_{2,} = { a_3}\,{ a_5},\medskip\\%
    \mathrm{Degree}\ 3:\ &i_{3} &= { a_0}\,{ a_5}\,{ a_7}+{ a_1}\,{ a_3}\,{
    a_8},\ %
    j_{3} = { a_0}\,{{ a_6}}^{2}+{{a_2}}^{2}{ a_8},\ \\%
    &k_{3} &= { a_1}\,{ a_5}\,{ a_6}+{ a_2}\,{ a_3}\,{ a_7},\ %
    l_{3} = { a_2}\,{{ a_5}}^{2}+{{a_3}}^{2}{ a_6},\medskip\\%
    \mathrm{Degree}\ 4:\ &i_{4} &= { a_0}\,{{ a_5}}^{2}{ a_6}+{ a_2}\,{{
    a_3}}^{2}{ a_8},\ %
    j_{4} = { a_0}\,{ a_3}\,{ a_6}\,{ a_7}+{ a_1}\,{ a_2}\,{ a_5}\,{ a_8},\ \\%
    &k_{4} &= { a_0}\,{ a_2}\,{{ a_7}}^{2}+{{ a_1}}^{2}{ a_6}\,{ a_8},\ %
    l_{4} = { a_1}\,{{ a_5}}^{3}+{{ a_3}}^{3}{ a_7},\ \\%
    &m_{4} &= { a_1}\,{ a_3}\,{{ a_6}}^{2}+{{ a_2}}^{2}{ a_5}\,{ a_7},\medskip\\%
    \mathrm{Degree}\ 5:\ &i_{5} &= {{ a_0}}^{2}{ a_6}\,{{ a_7}}^{2}+{{
    a_1}}^{2}{ a_2}\,{{ a_8}}^{2},\ %
    j_{5} = { a_0}\,{{ a_5}}^{4}+{{ a_3}}^{4}{ a_8},\ \\%
    &k_{5} &= { a_0}\,{{ a_3}}^{2}{{ a_7}}^{2}+{{ a_1}}^{2}{{ a_5}}^{2}{ a_8},\ %
    l_{5} = {{ a_1}}^{2}{{ a_6}}^{3}+{{ a_2}}^{3}{{ a_7}}^{2},\medskip\\%
    \mathrm{Degree}\ 6:\ &i_{6} &= {{ a_0}}^{2}{ a_3}\,{{ a_7}}^{3}+{{
    a_1}}^{3}{ a_5}\,{{ a_8}}^{2},\medskip\\%
    \mathrm{Degree}\ 7:\ &i_{7} &= {{ a_0}}^{3}{{ a_7}}^{4}+{{ a_1}}^{4}{{
    a_8}}^{3}.
    \end{array}
  \end{displaymath}
  Since there is an inclusion of invariant rings $$k[a_0, a_1, \ldots,
  a_8]^{\SL_2 (K)} \subset k[a_0, a_1, \ldots, a_8]^{D \cap \SL_2 (K)},$$ the
  Shioda invariants~\cite{shioda67} can be expressed as polynomials in the
  generic homogeneous dihedral invariants. For example, the degree 2 Shioda
  invariant can be written as
  \begin{equation*}
    {\frac {1}{70}}\,{{ i_1}}^{2}+2\,{ i_2}-\frac{1}{4}\,{ j_2}+\frac{1}{14}\,{
    k_2}-\frac{1}{28}\,{ l_2},
  \end{equation*}
  whereas the degree 3 invariant equals
  {\small $$
    {\frac {9}{34300}}\,{{ i_1}}^{3}+{\frac {3}{35}}\,{ i_1}\,{ i_2}
    +{\frac {9}{560}}\,{ i_1}\,{ j_2}-{\frac {33}{13720}}\,{ i_1}\,{
       k_2}-{\frac {27}{27440}}\,{ i_1}\,{ l_2}-{\frac {3}{56}}\,{ 
      i_3}+{\frac {9}{392}}\,{ j_3}-{\frac {3}{784}}\,{ k_3}+{\frac {9}{
        5488}}\,{ l_3} .
  $$}
  These formulas, as well as formulas expressing the dihedral invariants in
  terms of the Shioda invariants, are available online~\footnoteref{foothyp}.
\end{remark}

\subsection{Homogeneous dihedral invariants in the remaining
cases}\label{sec:other}

We now discuss the invariants that have to be used in the remaining cases.
First we treat
binary forms of type $(0,n,m)$ for odd $m$ (recall that in the previous
subsections we assumed $m$ to be even). Only small modifications are needed; the
generic homogeneous
diagonal invariants are given by
\begin{displaymath}
  \begin{array}{rcl}
    J_{2,0} & = & a_{2\ell-1} a_0 , \\
    J_{2,1} & = & a_{2\ell-2} a_1,\ 
    \ldots,\ 
    J_{2,\ell-1} = a_{\ell} a_{\ell-1} ,  \\
    J_4 & = & a_{\ell+1} a_{\ell-1}^3  \\
    & \vdots &  \\
    J_{2 \ell } & = & a_{2\ell-1} a_{\ell-1}^{2 \ell - 1}  .
  \end{array}
\end{displaymath}
These invariants suffice as long as $a_{2\ell - 1}$, $a_{2\ell-2}$, \ldots,
$a_{\ell+1}$, $a_{\ell-1}$ $\neq 0$. Symmetrizing with respect to the involution
$a_i \leftrightarrow a_{m - i}$, one again obtains the generic homogeneous
dihedral invariants for odd $m$. Homogeneous invariants for the non-generic
cases can be also constructed as for even $m$ as well; the only difference is
that the matrix $M_S$ is now given by $(2(s_1 - \ell)+1 , \ldots , 2(s_{r} -
\ell) + 1)$.

The homogeneous dihedral (and diagonal) invariants for binary forms of type
$(2,n,m)$ are exactly the same as expressions in the $a_i$ as for those of type
$(0,n,m)$. Finally, for the binary forms of type $(1,n,m)$, such, we only need
to consider the action of $T$ when constructing our invariants in light of the
second part of Proposition~\ref{prop:NormD}. But we know that the (identical)
homogeneous diagonal invariants considered for the types $(0,n,m)$ and $(2,n,m)$
already suffice to distinguish the orbits under this group. So we also know how
to construct a finite (and small) set of invariants for these curves.

\section{Explicit obstruction and descent}\label{sec:descent}

In this section, we will use the homogeneous dihedral invariants from Section
\ref{sec:dihedral} to obtain an explicit arithmetic description of the descent
obstruction for hyperelliptic curves with tamely cyclic reduced automorphism
group. If this obstruction vanishes, then we also indicate how an explicit
descent can be obtained. To phrase our results in a concise way, we first define
the \emph{type} of a hyperelliptic curve with tamely cyclic reduced automorphism
group.

\begin{definition}\label{def:type}
  Let $\Xm$ be a hyperelliptic curve over $K$. If $\Xm$ is isomorphic to a
  hyperelliptic curve associated with a binary form of type $(i,n,m)$ over $K$
  (as in Definition~\ref{def:typ}), then $\Xm$ will be said to be of
  \emph{type~$(i,n,m)$}. 
\end{definition}

\begin{remark}\label{rem:0nm}
  Let $\Xm$ be a hyperelliptic curve of type $(i,n,m)$, Then $i$ equals the
  number of Weierstrass points fixed by $G = \Aut (\Xm)$. The quantity $n$
  equals the cardinality of the reduced automorphism group of $\Xm$, since as
  in~\cite[Sec.1.2]{lr2012}, one can use the fact that the hyperelliptic
  involution is central to prove that the group $\overline{G}$ is canonically
  isomorphic with $\Aut (f)$.  Finally, if $n > 1$, then $m$ equals the
  cardinality of the divisor of branch points of $\Xm \to \Xm / G$ of order $2$.
  Conversely, any binary form of type $(i,n,m)$ determines a hyperelliptic curve
  with these geometric properties.

  Note that the genus $g$ of a hyperelliptic curve $\Xm$ of type $(i,n,m)$ is
  determined by the equality $2 g + 2 = m n + i$.
\end{remark}

In what follows, we let $\Xm$ denote a hyperelliptic curve over $K$ of type
$(i,n,m)$ whose field of moduli for the extension $K | k$ equals $k$.  In
Theorem~\ref{thm:MainTh1}, we have proved that the existence of a descent
implies the existence of a hyperelliptic descent except possibly if both $n$ and
$g$ are odd.  We now accordingly divide the issue of explicit descent into three
cases.
\begin{enumerate}[(i)]
  \item In the case where $n > 1$, Section~\ref{sec:hypdescobs} shows how to
    express the hyperelliptic descent obstruction in terms of the homogeneous
    dihedral invariants. Moreover, we discuss in Section~\ref{sec:exphypdesc}
    how to calculate a hyperelliptic descent explicitly if this obstruction
    vanishes.
  \item In the case where $n$ and $g$ are both odd,
    Section~\ref{sec:nonhypdescent} shows that the curve always descends,
    though perhaps not hyperelliptically. Moreover, we discuss how to calculate
    such a descent explicitly.
  \item In the case where $n = 1$, Section~\ref{sec:trivial} gives a
    generic method to calculate the (hyperelliptic) descent
    obstruction, and a corresponding descent if this obstruction vanishes. Its
    main approach is based on the covariant method developed in~\cite{lrs}. 
\end{enumerate}

To conclude these considerations, we show in Section~\ref{sec:counter}
how essentially all counterexamples to (hyperelliptic) descent can be
constructed.

\subsection{Explicit hyperelliptic descent obstruction}\label{sec:hypdescobs}

%In this and the following subsection, we suppose that the reduced automorphism
%group $\overline{G}$ is non-trivial and tamely cyclic.
%suppose that $X$ descends if and only if it descends hyperelliptically. Then
%generically, the descent problem can be solved effectively by the covariant
%method from~\cite{lrs}. I and focus from now on on the case where the reduced
%automorphism group $\overline{G}$ is non-trivial and tamely cyclic. We can then
%use homogeneous dihedral invariants to give criteria for the existence of a
%hyperelliptic descent and calculate the descent if this obstruction vanishes.
%An important proposition is the following.

In what follows, we will let $f$ be a binary form of type $(i,n,m)$ with $n >
1$. We denote homogeneous diagonal (resp.\ dihedral) invariants of $f$ by $J(f)$
(resp.\ $I(f)$).  We will often consider these tuples $J(f)$, $I(f)$ invariants
as points in the corresponding weighted projective spaces. As
in~\cite[Sec.1.3]{lr2012}, one can associate a unique representative with such a
point $p$, which we shall here call a \emph{normalized representative}. This
normalized representative is a tuple of coordinates that represents $p$ whose
entries are defined over the same field as the point $p$ when the latter is
considered as an element of a weighted projective space.

\begin{example}\label{ex:WPSNorm}
  Let $p = (3 : 6 \sqrt{3})$, considered as a point in the weighted projective
  $(2,3)$-space. Then $p$ is defined over $\QQ$, since its conjugate $(3 : -6
  \sqrt{3})$ can be obtained from $p$ by multiplying with the scalar $-1$.
  While the tuple $(3,6 \sqrt{3})$ representing $p$ is not defined over $\QQ$,
  its normalized representative from~\cite[Sec.1.3]{lr2012} is; this
  representative is given by $(\frac{1}{4} , \frac{1}{4})$.
\end{example}

Conversely, note that once a curve $\Xm$ over $K$ with tamely cyclic reduced
automorphism group is given explicitly, it is possible to quickly determine a
binary form of the corresponding type (as in~\eqref{ccform}),~\eqref{cform1}
or~\eqref{cform2}) that defines $\Xm$ over $K$ by using the methods
from~\cite[Sec.2]{lrs}. Indeed, using the methods in \emph{loc.\ cit.}\ one
diagonalizes the cyclic reduced automorphism group $\CG_n$ of $\Xm$ into our
standard embedding of the group $\CG_n$.

%\begin{proposition}\label{prop:ADInorm}
%  The normalized representative of the homogeneous dihedral invariants $I(f)$ of
%  $f$ is a $k$-rational tuple.
%\end{proposition}
%\begin{proof}
%  By hypothesis, the homogeneous dihedral invariants of $f$ and its
%  $\Gamma$-conjugates all define the same point in a weighted projective space.
%  It therefore suffices to invoke the uniqueness of the canonical representative
%  from~\cite[Sec.1.4]{lr2012}.
%\end{proof}
%
%\begin{proposition}\label{prop:f0}
%  The normalized representative of the homogeneous diagonal invariants $J(f)$ of
%  $f$ is defined over a quadratic extension $L = k (
%  \sqrt{d} )$ of $k$. Moreover, the binary form $f$ is isomorphic over $K$ to a
%  binary form $f_L$ of the same type that is defined over $L$.
%\end{proposition}
%\begin{proof}
%  The first part follows from the fact that the map between the corresponding
%  moduli space is of degree $2$, as follows for example from the fact that the
%  group $D$ in~\eqref{eq:D} contains the group of diagonal matrices as a
%  subgroup of index $2$. The second statement follows by
%  using the rational parametrization in Corollary~\ref{cor:param}.
%\end{proof}

\begin{proposition}\label{prop:f0}
  \begin{enumerate}[(i)]
    \item The normalized representative of the homogeneous dihedral invariants
      $I(f)$ of $f$ is defined over $k$.
    \item The normalized representative of the homogeneous diagonal invariants
      $J(f)$ of $f$ is defined over a quadratic extension $L = k ( \sqrt{d} )$
      of $k$.
    \item The binary form $f$ is isomorphic over $K$ to a binary form $f_L$ of
      the same type that is defined over $L$.
  \end{enumerate}
\end{proposition}
\begin{proof}
  (i) By Proposition~\ref{prop:NormD} the homogeneous dihedral invariants of $f$
  and its conjugates all define the same point in the corresponding weighted
  projective space, since by construction these invariants transform by suitable
  powers of a scalar under the action of $D$. It therefore suffices to invoke
  the uniqueness of the canonical representative from~\cite[Sec.1.4]{lr2012}.

%  (ii) This follows from the fact that the map between the corresponding moduli
%  spaces is of degree $2$, as can be proved by noting that the group $D$
%  in~\eqref{eq:D} contains the group of diagonal matrices as a subgroup of index
%  $2$.

  (ii) By Lemma~\ref{lem:C2Quot}, given a tuple of homogeneous dihedral
  invariants, there are at most $2$ tuples of homogeneous diagonal invariants of
  which these can be the symmetrization. As such, the Galois group fixing these
  tuples defines an at worst quadratic extension of $k$.

  (iii) One uses the rational parametrization in Corollary~\ref{cor:param}.
\end{proof}

\begin{definition}\label{def:invext}
  We call the field extension $L$ of $k$ in Proposition~\ref{prop:f0} the
  \emph{invariant extension} defined by $f$.
\end{definition}

\begin{corollary}\label{cor:descL}
  The curve $\Xm$ defined by $f$ descends to the at most quadratic invariant
  extension $L$ of $k$.
\end{corollary}

\begin{proposition}\label{prop:descL}
  Generically, the invariant extension $L$ is given by $k(\sqrt{d})$, where $d =
  I_{3,3,1}^2 - 4 I_{3,3,2}$ if $m$ is even and $d = I_{4,4,1}^2 - 4 I_{4,4,2}$
  if $m$ is odd.
\end{proposition}
\begin{proof}
  This follows because when $m$ is even (resp.\ odd) the field extension $L | k$
  is already incurred when reconstructing the first pair of non-dihedral
  diagonal invariants $J_3 , J'_3$ from $I_{3,3,1} , I_{3,3,2}$ (resp.\ $J_4 ,
  J'_4$ from $I_{4,4,1} , I_{4,4,2}$).
\end{proof}

We now consider some examples in order to get an idea of what the extension $L |
k$ looks like.

\begin{example}
  Consider the binary forms
  $$f = a_4 x^8 + a_3 x^6 z^2 + a_2 x^4 z^4 + a_1 x^2 z^6 + a_0$$
  of type $(0,2,4)$. Then the generic homogeneous dihedral invariants are
  symmetrizations of the generic diagonal invariants $J_1 = a_2$, $J_{2,0} = a_4
  a_0$, $J_{2,1} = a_3 a_1$ and $J_3 = a_4 a_1^2$.

  In this case the only new dihedral invariants obtained by symmetrizing the
  diagonal invariants are $I_{3,3,1} = J_3 + J_3'$ and $I_{3,3,2} = J_3 J'_3$. For the
  generic forms $f$ in Proposition~\ref{prop:GenTInv} of type $(0,2,4)$, the
  quadratic extension $L$ is therefore \emph{always} the one incurred by passing
  from $J_3 + J_3'$ and $J_3 J'_3$ to $J_3 , J'_3$. As we have seen, this means
  that $L = k(\sqrt{d})$, where $d = I_{3,3,1}^2 - 4
  I_{3,3,2}$.
\end{example}

\begin{example}
  Consider the binary forms $$f = a_5 x^{10} + a_4 x^8 z^2 + a_3 x^6 z^4 + a_2
  x^4 z^6 + a_1 x^2 z^8 + a_0 z^{10}$$ of type $(0,2,5)$. Then the generic
  homogeneous dihedral invariants are symmetrizations of $J_{2,0} = a_6 a_0$,
  $J_{2,1} = a_5 a_1$, $J_{2,2} = a_4 a_2$, $J_4 = a_5 a_2^3$ and $J_6 = a_6
  a_2^5$.
  
  This time the quadratic extension $L$ is a bit more complicated to determine.
  Indeed, we get two pairs of new dihedral invariants, namely $I_{4,4,1},
  I_{4,4,2}$ and $I_{6,6,1}, I_{6,6,2}$. Generically, the extension $L$ is
  already incurred by passing from $I_{4,4,1}, I_{4,4,2}$ to $J_4, J'_4$, which
  gives $L = k (\sqrt{d})$ where $d = I_{4,4,1}^2 - 4 I_{4,4,2}$. But it is
  possible that this does not give an extension of the ground field, while
  passing from $I_{6,6,1}, I_{6,6,2}$ to $J_6, J'_6$ does. In the latter case we
  have $L = k (\sqrt{d})$ with $d = I_{6,6,1}^2 - 4 I_{6,6,2}$ instead.
\end{example}

Now let $L$ be the invariant extension defined by $f$, and let $f_L$ be the
partial descent from Proposition~\ref{prop:f0}. We may suppose that $\Xm$
defined by $f_L$. As at the beginning of Section~\ref{sec:hypdescrit}, the
isomorphisms between $\Xm$ and its conjugates induce a canonical descent datum
on the quotient $\Bm = \Xm / \Aut (\Xm)$, which yields a model $\Bm_0$ of $\Bm$
over $k$. Now Theorem~\ref{thm:MainTh2} shows that $\Xm$ descends
hyperelliptically if and only if $\Bm_0$ has a $k$-rational point.

To study the twist $\Bm_0$, we construct the corresponding Weil cocycle $c$. Let
$\sigma$ be the generator $\Gal (L | k)$.  By our running hypotheses,
$f_L^{\sigma}$ has the same homogeneous dihedral invariants as $f_L$. Let $S$ be
the matrix $\smat{0}{1}{1}{0}$. Then either
\begin{equation}\label{eqtranss0}
  f_L^{\sigma} \sim D . f_L
\end{equation}
or
\begin{equation}\label{eqtranss}
  f_L^{\sigma} \sim D S . f_L
\end{equation}
for some $D = \smat{\lambda}{0}{0}{\mu} \in T$. Note that by
Proposition~\ref{prop:NormD}, the latter case does not occur if $f_L$ has type
$(1,n,m)$. Regardless, we now either have
\begin{small}
  \begin{equation}\label{cctranss0}
    (a^{\sigma}_{m} , a^{\sigma}_{m - 1} , \ldots , a^{\sigma}_1 ,
    a^{\sigma}_0) \longmapsto (\lambda^{m n} a_{m} , \lambda^{(m - 1) n} \mu^n
    a_{m - 1} , \ldots , \lambda^n \mu^{(m - 1) n} a_1 , \mu^{m n} a_0 )
  \end{equation}
\end{small}
or
\begin{small}
  \begin{equation}\label{cctranss}
    (a^{\sigma}_{m} , a^{\sigma}_{m - 1} , \ldots , a^{\sigma}_1 ,
    a^{\sigma}_0) \longmapsto (\lambda^{m n} a_{0} , \lambda^{(m - 1) n} \mu^n
    a_{1} , \ldots , \lambda^n \mu^{(m - 1) n} a_{m-1} , \mu^{m n} a_{m} ) .
  \end{equation}
\end{small}
depending on whether~\eqref{eqtranss0} or~\eqref{eqtranss} holds.

%The fact that the form $f_0$ does not allow any more automorphisms than $\CG_n$
%allows us to conclude that the exponents of the powers of $\lambda$ and $\mu$
%occurring at the nonzero coefficients of $f_0$ both have greatest common divisor
%$n$. Using invariance of the arithmetic diagonal invariants then allows us to
%conclude that $\lambda^n$ and $\mu^n$ are both in $k$. 

\begin{lemma}\label{lem:trivcocyc}
  Choose isomorphisms $g_{\sigma} : \Xm \to \Xm^{\sigma}$ for all $\sigma \in
  \Gamma$. Then the induced Weil cocycle $c$ on $\Bm$ given by $\sigma \mapsto
  h_{\sigma} : \Bm \to \Bm^{\sigma}$ is trivial on the index $2$ subgroup of
  $\Gamma$ that fixes the invariant extension $L$ of $k$.
\end{lemma}
\begin{proof}
  Indeed, since we divide out by the automorphisms of $\Xm$, the induced maps
  $h_{\sigma}$ are independent of the choice of the $g_{\sigma}$. Since $\Xm$ is
  defined over $L$, we may just take $g_{\sigma}$ to be the identity if
  $\sigma$ fixed $L$. The result follows.
\end{proof}

Using the inflation-restriction exact sequence, Lemma~\ref{lem:trivcocyc}
implies that $c \in H^1 (\Gal (K | k) , \PGL_2 (K))$ is the inflation of a Weil
cocycle $c_L \in H^1 (\Gal (L | k) , \PGL_2 (L))$. In the case
\eqref{eqtranss0}, the cocycle $c_L$ is given by
\begin{equation}\label{cocyc1}
  \sigma \longmapsto \smat{\lambda^n}{0}{0}{\mu^n}
\end{equation}
and in the case \eqref{eqtranss} by
\begin{equation}\label{cocyc2}
  \sigma \longmapsto \smat{0}{\mu^n}{\lambda^n}{0}\,.
\end{equation}

Suppose that $c_L$ is given by~\eqref{cocyc1}. Then by dividing by the scalar
$\lambda^n$, we can normalize $c_L$ to
\begin{equation}\label{cocyc3}
  \sigma \longmapsto \smat{1}{0}{0}{r}\,. 
\end{equation}
The Weil cocycle condition translates into the equality $r^{\sigma} r = 1$, so
by Hilbert 90 the cocycle~\eqref{cocyc3} is a coboundary. More precisely, the
descent morphism is then given by a diagonal matrix, so there exists a
hyperelliptic descent defined by a binary form $f_k$ of the same type as $f_L$.
But in that case the normalized representative of $I(f_L)$ would be defined over
$k$ already; so $L$ was already the trivial extension of $k$. Since we are
always in the case~\eqref{cocyc1} if $f_L$ is of type $(1,n,m)$, we get the
following result.

\begin{lemma}\label{lem:1nm}
  If $\Xm$ is of type $(1,n,m)$, then $\Xm$ can be defined over $k$ by a binary
  form $f$ of the given type.
\end{lemma}

%, a representative of which
%over the field of moduli can be reconstructed by normalizing the homogeneous
%diagonal invariants and using the parametrization from
%Corollary~\ref{cor:param}.

In the second case that $c_L$ is given by~\eqref{cocyc2}, let $r = \mu^n /
\lambda^n$. Now $c_L$ normalizes to
\begin{equation}\label{cocyc4}
  \sigma \longmapsto \smat{0}{r}{1}{0}\,. 
\end{equation}
The fact that~\eqref{cocyc2} indeed defines a cocycle shows that $r^{\sigma} =
r$, so $r \in k$.

\begin{definition}\label{def:normobs}
  We call the image of $r$ in the quotient group $k^* / \Nm_{L | k} (L^*)$ the
  \emph{norm obstruction} for $\Xm$.
\end{definition}

\begin{lemma}\label{lem:rgen}
  Let $f_L$ be a form of type $(0,n,m)$ or $(2,n,m)$.
  \begin{enumerate}[(i)]
    \item If $m = 2 \ell$ is even, then if we suppose additionally that $f_L$ is
      generic, then the norm obstruction for $\Xm$ is trivial if and
      only if the generic homogeneous dihedral invariant $I_{2,\ell - 1}(f)$ is
      a norm from $L$.
    \item If $m$ is odd, then the norm obstruction is always trivial.
  \end{enumerate}
\end{lemma}
\begin{proof}
  (i) If $a_{\ell-1}$, $a_{\ell}$ and $a_{\ell+1}$ are all nonzero, the
  transformation formula~\eqref{cctranss} shows that we have
  \begin{equation*}
    r = (a_{\ell}^{\sigma} a_{\ell-1}) / (a_{\ell+1}^{\sigma} a_\ell) =
    a_{\ell-1} / a_{\ell+1}^{\sigma}
  \end{equation*}
  (note that $a_{\ell} = a_{\ell}^{\sigma}$ by the $k$-rationality of the
  homogeneous dihedral invariants). The demand that this be a norm is satisfied
  if and only if $a_{\ell + 1} a_{\ell - 1} = I_{2 , \ell-1}$ is a norm.

  (ii) Let $m = 2 \ell - 1$ be odd. We first suppose that $a_{\ell}$ and
  $a_{\ell + 1}$ are nonzero. Then $r = (a^{\sigma}_{\ell+1} a_{\ell+1}) /
  (a^{\sigma}_{\ell} a_{\ell})$ is a norm. In the general case, the same
  argument shows that $r^{2 i -1}$ is a norm for all $i$ such that $a_{\ell+i}$
  (and hence $a_{\ell+1-i}$, since we are in case~\eqref{cocyc2}) is nonzero.
  The set of exponents of $r$ thus obtained has greatest common divisor equal to
  one, since one observes that otherwise the binary form $f$ that we started with
  would have more automorphisms than $\CG_n$ and would therefore not be of the
  given type.
\end{proof}

We can now prove our main theorem of this section.

\begin{theorem}\label{thm:MainTh3}
  Let $\Xm$ denote a hyperelliptic curve over $K$ of genus $g$ and type
  $(i,n,m)$ with $n > 1$ whose field of moduli for the extension $K | k$ equals
  $k$, represented by a binary form $f$ over $K$ of the given type. Let $L | k$
  be the invariant extension defined by $f$. Then the hyperelliptic descent
  obstruction is as follows, depending on the type of $\Xm$.
  \begin{itemize}
    \item If $\Xm$ is of type $(0,n,m)$ or $(2,n,m)$, then a hyperelliptic
      descent always exists if $m$ is odd. If $m$ is even, then $\Xm$ descends
      hyperelliptically if and only if its norm obstruction is trivial. In
      either of the two cases, $\Xm$ always admits a hyperelliptic model of the
      given type over the at most quadratic extension $L$ of $k$.
    \item If $\Xm$ is of type $(1,n,m)$, then a hyperelliptic descent always
      exists. Moreover, this descent can be defined by a hyperelliptic model of
      the given type over $k$.
  \end{itemize}
\end{theorem}
\begin{proof}
  By Proposition~\ref{prop:f0}(ii), we can always construct a hyperelliptic
  model of $\Xm$ over the quadratic extension $L$ of $k$. By
  Lemma~\ref{lem:1nm}, this extension $L$ in fact coincides with $k$ if $f$ is
  of type $(1,n,m)$, which proves the theorem for this case.

  It remains to see when the descent obstruction to $k$ vanishes in the other
  cases. By Theorem~\ref{thm:MainTh2}, this is the case if and only if the
  canonical descent $\Bm_0$ of $\Bm = \Xm / \Aut (\Xm)$ admits a point over $k$.
  The twist $\Bm_0$ of the projective line $\Bm$ is determined by the cocycle
  $\sigma$ in~\eqref{cocyc4}. Proposition~\ref{prop:WeilTest} now shows that
  $\Bm_0$ is isomorphic to $\PP^1$ if and only if the norm obstruction for $\Xm$
  vanishes.  It now suffices to invoke Lemma~\ref{lem:rgen}(ii) to show that the
  extension $L | k$ is always trivial if $m$ is odd.
\end{proof}

%\begin{remark}
%  As we have seen at the beginning of Section~\ref{sec:dihedral}, the type of the
%  curve $\Xm$ determines not merely its genus $g$, but also its automorphism
%  group $G$. Conversely, the descent obstruction in Theorem~\ref{thm:MainTh3}
%  can in fact be expressing purely in terms of $g$ and $G$. However, this leads
%  to rather messy case distinctions, which is why we have decided to use the
%  current formulation.
%\end{remark}

\begin{remark}\label{rem:sig}
  The existence of a descent can sometimes also be proved by
  using~\cite{brandt-stich} and a signature argument as
  in~\cite[Prop.4.3]{lr2012} to show that $\Bm_0$ has a $k$-rational point. That
  there exist a hyperelliptic descent then follows from
  Theorem~\ref{thm:MainTh2}. However, our explicit construction of $f_0$ in the
  following section uses the homogeneous diagonal invariants and the
  parametrization from Corollary~\ref{cor:param} in an essential way.
\end{remark}

\subsection{Explicit hyperelliptic descent}\label{sec:exphypdesc}

We will now show how to construct a descent of $\Xm$ to $k$ if the obstruction
in Theorem~\ref{thm:MainTh3} vanishes. For this, we first prove the following
proposition.

\begin{proposition}\label{prop:cycP1}
  Let $\dD_0$ be a $k$-rational effective divisor of degree $2$ on $\PP^1$. Then
  for every $n > 1$ prime to the characteristic of $k$ there exists a tamely
  cyclic cover $\PP^1 \to \PP^1$ of degree $n$ that is defined over $k$ and
  whose branch divisor has support in $\dD_0$.
\end{proposition}
\begin{proof}
  The case where $\dD_0 = [p_1] + [p_2]$ with $p_1, p_2 \in k$ is trivial. In the
  case where $p_1$ and $p_2$ are Galois conjugate, we can change coordinates in
  $\PP^1$ to suppose that $\dD_0 = [\sqrt{d}] + [-\sqrt{d}]$, where $d$ is
  non-square in $k$. In that case, consider the expansion of the expression $(x
  + \sqrt{d}z)^n$ as $p + q \sqrt{d}$, with $p,q \in k[x,z]$. Then we claim that
  we can take $(x:z) \mapsto (p : q)$ as our cover.

  To see this, first note that $p$ and $q$ do not contain a common factor.
  Indeed, this would be a factor of $(x + \sqrt{d} z)^n$ as well, hence it would
  equal $(x + \sqrt{d} z)$. But because $p$ and $q$ are defined over $k$, they
  would then both be divisible by $(x^2 - d z^2)$. Hence the same would be true
  for $(x + \sqrt{d} z)^n$, which is absurd. So $(p : q)$ does indeed define a
  degree $n$ cover of $\PP^1$ over $k$.

  To see that $(p : q)$ is (tamely) cyclic, note that by construction,
  the equation $p(t,1)/q(t,1) = - \sqrt{d}$ has $t = -\sqrt{d}$ as an $n$-fold
  solution.  Therefore $(-\sqrt{d} : 1)$ is in the branch locus of $(p : q)$,
  and hence $\sqrt{d}$ as well since $(p : q)$ is defined over $k$. The
  Riemann-Hurwitz formula excludes the possibility of other points occurring in
  the branch locus of $(p : q)$, which is therefore indeed given by $\dD_0$.
\end{proof}

We consider the first case of Theorem~\ref{thm:MainTh3}. In order to descend
effectively, we first construct some special divisors on the canonical model
$\Bm_0$ of $\Bm = X / \Aut (X)$.

We let $\rR$ be the support of the branch divisor of the quotient map $\pi : \Xm
\to \Bm$. Given $\sigma \in \Gal (K | k)$, the divisor $\rR$ is mapped to its
conjugate $\rR^{\sigma}$ under the well-determined isomorphisms $\Bm \to
\Bm^{\sigma}$ induced by a choice of isomorphism $\Xm \to \Xm^{\sigma}$.  We let
$\rR_0$ be the image of $\rR$ under the canonical descent morphism $\varphi : 
\Bm \to \Bm_0$.

The branch divisor $\rR$ naturally admits a decomposition $\rR = \sS + \tT$ into
effective subdivisors, $\sS$ and $\tT$. Here $\tT$ is the branch divisor of the
tamely cyclic cover $q : \Qm = \Xm / \iota \to \Bm$, and $\sS$ is contained in
the image of the branch divisor of the quotient morphism $\pi_{\iota} : \Xm \to
\Qm$ under $q$. We let $\sS_0$ (resp.\ $\tT_0$) be the image of $\sS$ (resp.\
$\tT$) under $\varphi$.

We summarize the situation, as well as indicating some additional divisors which
we will obtain later in our argument, in the diagram below.
\begin{equation*}
  \xymatrix{
    & \dD \ar@{}[d]|-*[@]{\subset} & \rR = \sS + \tT \ar@{}[d]|-*[@]{\subset} \\
      \Xm \ar[r]^{\pi_{\iota}} & \Qm = \Xm / \iota \ar[r]^{q} \ar[d]
    & \Bm = \Xm / \Aut (\Xm) \ar[d] \\
    & \Qm_0 \ar[r]^{q_0} & \Bm_0 \\
    & \dD_0 \ar@{}[u]|-*[@]{\subset} & \rR_0 = \sS_0 + \tT_0
      \ar@{}[u]|-*[@]{\subset}
  }
\end{equation*}

\begin{proposition}
  \begin{enumerate}[(i)]
    \item The divisors $\rR_0$, $\sS_0$ and $\tT_0$ are defined over $k$.
    \item The support of $\tT_0$ is of degree $2$.
  \end{enumerate}
\end{proposition}
\begin{proof}
  (i) This follows because the action of an element $\sigma \in \Gamma$
  transforms the branch divisor of $\pi : \Xm \to \Bm$ (resp. $q : \Xm /
  \iota_{\Xm} \to \Bm$) into the branch divisor of $\pi^{\sigma} : \Xm^{\sigma}
  \to \Bm^{\sigma}$ (resp.\  $q^{\sigma} : \Xm^{\sigma} / \iota_{\Xm}^{\sigma}
  \to \Bm^{\sigma}$). Note that for $\Xm / \iota \to \Bm$ this uses the fact
  that the involution $\iota$ is canonical to obtain the equality
  $\iota_{\Xm}^{\sigma} = \iota_{\Xm^{\sigma}}$.

  (ii) Taking a normal form~\eqref{ccform}-\eqref{cform2} over the algebraic
  closure $K$ transforms the quotient map $q$ into the map $(x:z) \mapsto (x^n
  ,z^n)$, for which $\tT$ becomes the divisor $(n-1) [0] + (n-1) [\infty]$.
\end{proof}

We first consider the curves $\Xm$ defined by a form $f$ of type $(0,n,m)$ or
$(2,n,m)$. The quotient $\Bm$ has natural coordinates $(s:t) = (x^n : z^n)$, in
terms of which $\tT$ is given by the zero locus of $a_m s^m + a_{m-1} s^{m-1} t
+ \ldots + a_1 s t^{m-1} + a_0 t^m$. If the hyperelliptic descent obstruction
vanishes, then $\Bm_0$ is isomorphic with $\PP^1$ over $k$, and we can apply the
explicit matrix $N$ from the proof of Proposition~\ref{prop:WeilTest} to $\tT$
to get the $k$-rational divisor $\tT_0$ on $\Bm_0 \cong \PP^1$. The divisor $\sS$,
which corresponds to $(n - 1)[(1:0)] + (n - 1)[(0:1)]$ in our normalization, is transformed
under $N$ to the $k$-rational divisor $\sS_0 = (n - 1) [(1 : \beta)] + (n - 1) [(1 :
\beta^{\sigma})]$. We can now apply Proposition~\ref{prop:cycP1} with
$\dD$ equal to the support $\uU_0$ of $\tT_0$ to get a model $q_0 : \Qm_0  :=
\PP^1 \to \PP^1 = \Bm_0$ defined over $k$ of the quotient map $q$. We now
distinguish three cases:

\begin{enumerate}
  \item If $\Xm$ has type $(0,n,m)$, then the branch divisor $\dD$ of
    $\pi_{\iota}$ equals the pullback $q^* (\sS)$. Pulling back $\sS_0$ by
    $q_0$, we therefore get a $k$-rational model $\dD_0 = q_0^* (\sS_0)$ on
    $\Qm_0 \cong \PP^1$ of this branch divisor. We can then construct a model
    $\Xm_0$ of $\Xm$ over $k$ by taking the degree $2$ cover of $\Qm_0 \cong
    \PP^1$ ramified over $\dD_0$.
  \item If $\Xm$ has type $(2,n,m)$, the pullback $q^* (\sS)$ is properly
    contained in $\dD$; we have to add the two points $(1:0)$ and $(0:1)$ in the
    ramification locus of $q$ that constitute the support $\uU$ of $\tT$. This support transforms into the
    ramification divisor $\uU_0 = [(1 : \beta)] + [(1 : \beta^{\sigma})]$ of $q_0$,
    which is $k$-rational. So by ramifying over $\dD_0 = q_0^* (\sS_0) + \uU_0$
    instead, we again get a hyperelliptic descent.
  \item If $\Xm$ has type $(1,n,m)$, then Lemma~\ref{lem:1nm} shows that the
    construction of the binary form $f_L$ from the normalized diagonal
    invariants of $\Xm$ in fact automatically gives rise to a form $f_0 = f_L$
    defined over $k$.
\end{enumerate}

Combining these three cases, we get the following algorithm to construct a
hyperelliptic descent if the obstruction vanishes.

\begin{algo}\label{alg:hypdesc}
  Let $\Xm$ denote a hyperelliptic curve over $K$ of genus $g$ and type
  $(i,n,m)$ with $n > 1$ whose field of moduli for the extension $K | k$ equals
  $k$. Suppose that the hyperelliptic descent obstruction vanishes for $\Xm$.
  Then a binary form $f_0$ defined over $k$ that gives a hyperelliptic
  descent $\Xm_0 : y^2 = f_0$ of $\Xm$ can be constructed as follows.
  \begin{enumerate}[(i)]
    \item Using the methods in~\cite[Sec.2]{lrs}, construct a binary form $f$ of
      type $(i,n,m)$ that represents $\Xm$.
    \item Compute the normalized homogeneous dihedral invariants $I(f)$.
    \item Construct a descent $f_L$ to the invariant extension $L$ of $k$
      defined by $f$ by using Corollary~\ref{cor:param}.
    \item If $i = 1$, then set $f_0 = f_L$ and terminate.
    \item Determine the quantity $r$ in~\eqref{cocyc4} for $f_L$, either by
      using Lemma~\ref{lem:rgen} in the generic case, or alternatively by using
      the methods in~\cite[Sec.2]{lrs}.
    \item Determine a coboundary matrix $N =
      \smat{1}{\lambda^{\sigma}}{\beta}{\lambda^{\sigma} \beta^{\sigma}}$ as in
      Proposition~\ref{prop:WeilTest}.
    \item Let $\uU_0 = [(1 : \beta)] + [(1 : \beta^{\sigma})]$. Construct the
      $k$-rational morphism $q : \PP^1 \to \PP^1$ ramifying over $\uU_0$ as in
      Proposition~\ref{prop:cycP1}.
    \item Let $\dD_0 = q_0^* (\tT_0)$ (resp.\ $\dD_0 = q_0^* (\sS_0) + \uU_0$)
      if $i = 0$ (resp.\ $i = 2$).
    \item Let $f_0$ be the monic polynomial in
      $k[x]$ whose zero divisor equals $\dD_0$. Return $f_0$ and terminate.
  \end{enumerate}
\end{algo}

We refer to Example~\ref{ex:desc1} for a concrete calculation with this
algorithm.

\subsection{Explicit non-hyperelliptic descent}\label{sec:nonhypdescent}

In the case where asking for a hyperelliptic descent and a general descent is
not equivalent, it turns out that one can always descend.

\begin{theorem}\label{thm:MainTh4}
  Let $\Xm$ denote a hyperelliptic curve over $K$ of genus $g$ and type
  $(i,n,m)$ with $n > 1$ whose field of moduli for the extension $K | k$ equals
  $k$. If $n$ and $g$ are odd, then $\Xm$ descends.
\end{theorem}
\begin{proof}
  Let $f$ be a binary form representing $\Xm$ of the given type. As in
  Proposition~\ref{prop:f0}, we first construct a descent $f_L$ of $f$ to the
  invariant extension $L$ of $k$. And once more, as in
  Section~\ref{sec:hypdescobs} the study of the cocycle $c_L \in H^1 (\Gal (L |
  k) , \PGL_2 (L))$ given by~\eqref{cocyc4} will be crucial.

  Let us first explicitly construct a conic $\Qm$ corresponding to the cocycle
  $c_L$. We take $\Qm$ to be given by the equation $x^2 + \lambda y^2 + \mu z^2 =
  0$, where $\lambda = - 1 / r$ and where $\mu \in k$ is such that $L =
  k(\sqrt{-\mu})$. Consider the $L$-rational morphism $\varphi : \PP^1 \to \Qm$
  given by the rational parametrization from the point $(\sqrt{-\mu} : 0 : 1)
  \in \Qm(L)$. Then one verifies that $\varphi^{\sigma} = \varphi \alpha$ for the
  automorphism $\alpha : x \mapsto r / x$ of $\PP^1$ corresponding
  to~\eqref{cocyc4}.

  So $\Qm$ is isomorphic to the canonical model $\Bm_0$ of $\Bm = \Xm / \Aut
  (\Xm)$ over $k$. Moreover $\varphi$ can be used as a descent morphism $\Bm \to
  \Bm_0$. This morphism transforms the branch divisor $\tT = [(1:0)] + [(0:1)]$ of
  the quotient morphism $q : \Qm = \Xm /\iota \to \Xm / \Aut (\Xm) = \Bm$ into a
  $k$-rational divisor $\tT_0$. Indeed, we have
  \begin{equation*}
    \tT_0^{\sigma} = (\varphi_* ([0] +[ \infty]))^{\sigma} = \varphi_*^{\sigma}
    ([0] +[ \infty]) = \varphi_* (\alpha_* ( [0] +[ \infty] )) = \varphi_*
    ([\infty] + [0]) = \tT_0  .
  \end{equation*}

  This allows us to once more construct a cyclic cover $q_0 : \Qm \to \Qm$ ramifying
  over $\tT_0$ that is a $k$-rational model of the cyclic cover $q : \Qm \to \Bm$
  ramifying over $\tT$. Indeed, let $f : \PP^1 \to \PP^1$ be the $K$-rational
  morphism given by $x \to x^n / r^{(n-1)/2}$, and let $f_0 = \varphi f
  \varphi^{-1}$. One verifies that $f = \alpha f \alpha^{-1}$, which implies
  that $f_0 = \varphi f \varphi^{-1} = \varphi \alpha f \alpha^{-1} \varphi^{-1}
  = \varphi^{\sigma} f (\varphi^{\sigma})^{-1} = f_0^{\sigma}$.  Therefore we
  can take $q_0 = f_0$.

  If $\Xm$ has type $(0,n,m)$, then we once again get a $k$-rational model
  $\dD_0 = q_0^* (\sS_0)$ of the branch divisor $\dD$ of $\pi_{\iota}$, this
  time on the conic $\Qm$, which is not necessarily isomorphic with $\PP^1$ over
  $k$. If $\Xm$ has type $(2,n,m)$, then we again have to throw in the
  ramification divisor $\uU_0$ of $q_0$ with $q_0^* (\sS_0)$ to get $\dD_0$.
  Note that this ramification divisor is again $k$-rational; in fact it is given
  by the zero divisor $(y)_0$ of the function $y$ on $\Qm$.

  Regardless, one can now construct a $k$-rational degree~$2$ cover $\Xm_0$ of
  $\Qm$ that ramifies over $\dD_0$ as in the proof~\cite[Prop.4.13]{lr2012}, since
  $g$ is odd. This $\Xm_0$ is the desired descent.
\end{proof}

In this case, the algorithm to obtain a descent is as follows.

\begin{algo}\label{alg:nonhypdesc}
  Let $\Xm$ denote a hyperelliptic curve over $K$ of genus $g$ and type
  $(i,n,m)$ with $n > 1$ whose field of moduli for the extension $K | k$ equals
  $k$. Suppose that $n$ and $g$ are both odd. Then a descent $\Xm_0$ of $\Xm$
  can be constructed as follows.
  \begin{enumerate}[(i)]
    \item Using the methods in~\cite[Sec.2]{lrs}, construct a binary form $f$ of
      type $(i,n,m)$ that represents $\Xm$.
    \item Compute the normalized homogeneous dihedral invariants $I(f)$.
    \item Construct a descent $f_L$ to the invariant extension $L$ of $k$
      defined by $f$ by using Corollary~\ref{cor:param}.
    \item If $i = 1$, then set $f_0 = f_L$ and terminate.
    \item Determine the quantity $r$ in~\eqref{cocyc4} for $f_L$, either by
      using Lemma~\ref{lem:rgen} in the generic case, or alternatively by using
      the methods in~\cite[Sec.2]{lrs}.
    \item Let $\lambda = -1/r$ and let $\mu \in k$ be such that $L =
      k(\sqrt{\mu})$. Construct the conic $\Qm : x^2 + \lambda y^2 + \mu z^2 = 0$
      over $k$, and let $\varphi : \PP^1 \to \Qm$ be the rational parametrization
      from the point $(\sqrt{-\mu} : 0 : 1) \in \Qm(L)$.
    \item Let $f : \PP^1 \to \PP^1$ be the $K$-rational morphism given by $x \to
      x^n / r^{(n-1)/2}$, and let $f_0 = \varphi f \varphi^{-1}$.
    \item Let $\dD_0 = f_0^* (\sS_0)$ (resp.\ $\dD_0 = f_0^* (\sS_0) +
      (y)_0$) if $i = 0$ (resp.\ $i = 2$).
    \item As in~\cite[Prop.4.13]{lr2012}, let $\Xm_0$ be the $k$-rational
      degree~$2$ cover of $\Qm$ ramifying in $\dD_0$.
  \end{enumerate}
\end{algo}

A calculation involving this algorithm can be found in Example~\ref{ex:condesc}.

\subsection{The case of trivial reduced automorphism group}\label{sec:trivial}

We conclude our discussion of explicit descent by considering the case where the
hyperelliptic curve $\Xm$ over $K$ is of type $(i,1,m)$, or more
straightforwardly expressed, the reduced automorphism group $\Xm$ is trivial.
Our descent obstruction results in the previous sections generalize to this
case. If $g$ is even, then our Theorem~\ref{thm:MainTh1} recovers a classical
result by Mestre~\cite{mestre} which states that in even genus a curve $\Xm$ with trivial
reduced automorphism group descends if and only if it descends hyperelliptically
\cite{mestre}. On the other hand, if $g$ is odd, then~\cite[Prop.4.13]{lr2012}
shows that a descent always exists, completely in line with our
Theorem~\ref{thm:MainTh4}.

Still, to construct an explicit descent $\Xm_0$ of $\Xm$ in these cases is
actually more complicated, due to the absence of homogeneous dihedral
invariants. We briefly discuss two ways to get around this problem.

\subsubsection{The covariant method}

The first and most effective way is to use the \emph{covariant
method}~\cite{lrs}. We now apply it to the case under consideration.

\begin{proposition}\label{prop:covdesc} Let $\Xm$ be a hyperelliptic curve with
  trivial reduced automorphism group, defined by a binary form $f$. Let $c$ be a
  covariant of $f$ with single roots whose automorphism group is trivial as
  well, and let $\Ym : y^2 = c$ be the hyperelliptic curve defined by $c$.
  \begin{enumerate}[(i)]
    \item The field of moduli of the curve $\Ym$ with respect to the extension
      $K | k$ again equals $k$.
    \item $\Xm$ admits an hyperelliptic descent if and only if $\Ym$ does.
      Moreover, if $\Ym$ does not allow a hyperelliptic descent, then neither
      does $\Xm$ allow a general descent if the genus of $\Xm$ is even.
    \item Suppose that $\Ym$ admits a hyperelliptic descent $\Ym_0$ defined by a
      homogeneous polynomial $c_0$. Then if $A \in \GL_2 (K)$ transforms $c$
      into $c_0$, the transformation $A . f$ of $f$ also yields a descent of $f$
      (after possibly dividing out a scalar).
    \item Suppose that the genus of $\Xm$ is odd. Let $\Rm = \Ym / \iota_{\Ym} =
      \Ym / \Aut (\Ym)$, and let $\Rm_0$ be the canonical model of $\Rm$. Let
      $\varphi : \Rm \to \Rm_0$ be the canonical descent morphism. Let $\dD$ be
      the branch locus of $\pi : \Xm \to \Xm / \iota_{\Xm} = \Xm / \Aut (\Xm)$.
      Then the image $\dD_0 = \varphi (\dD)$ is a $k$-rational divisor on
      $\Bm_0$. There exists a degree $2$ cover $\Xm_0$ of $\Rm_0$ over $k$ whose
      branch locus equals $\dD_0$. The curve $\Xm_0$ is then a descent of $\Xm$.
  \end{enumerate}
\end{proposition}
\begin{proof}
  (i) By definition of covariance, the isomorphisms $\Xm \to \Xm^{\sigma}$ give
  rise to isomorphisms $\Ym \to \Ym^{\sigma}$.

  (ii) The first part follows from~\cite[Thm.3.8]{lrs}, and the second part
  from~\cite{mestre}.

  (iii) This again follows from~\cite[Thm.3.8]{lrs}. 

  (iv) The canonical descent datum on the quotient $\Rm$ gives rise to the
  conic $\Rm_0$, which is a $k$-rational model for both $\Ym / \iota_{\Ym}$ and
  $\Xm / \iota_{\Xm}$. By covariance, the morphism $\varphi$ is also the Weil
  coboundary $(\Xm / \iota_{\Xm} , \dD) \to (\Rm_0 , \dD_0)$ for the pair $(\Xm /
  \iota_{\Xm} , \dD)$. Therefore the image $\dD_0$ is indeed $k$-rational. One
  then again invokes~\cite[Prop.4.13]{lr2012}.
\end{proof}

\begin{remark}
  Proposition~\ref{prop:covdesc} is especially useful for sextic and octavic
  covariants $c$, since for these, the results from~\cite{mestre}
  and~\cite[Sec.2]{lrs} allow us to test effectively whether it has trivial
  automorphism group. Moreover, in these cases effective methods to determine
  the descent obstruction are available, as well as methods to determine an
  explicit descent if this obstruction vanishes.
\end{remark}

%We can then still explicitly compute a descent conic $\Qm$ and $K$-isomorphism
%$\varphi$ for the descent datum on $\Cm$. In other words, the $K$-isomorphism
%$\varphi$ is a coboundary for the cocycle of K-isomorphisms between $\Cm$ and
%its Galois conjugates $\Cm^{\sigma}$, hence also for the cocycle of
%K-isomorphisms between $\Xm$ and its conjugates $\Xm^{\sigma}$. Therefore we
%can descend the branch locus of $\Xm/\iota$ to a $k$-rational divisor on $\Qm$ by
%applying $\varphi$. Since the genus of $\Xm$ is odd, this suffices to descend
%$\Xm$, as in the proof of~\cite[Prop.4.13]{lr2012}.
  
%In the former case, one can use the results from~\cite{mestre} to descend the
%resulting covariant curve of genus $2$; this will also yield a descent for
%$\Xm$ itself by~\cite[Theorem 2.8]{lrs}. In the latter case, one uses the
%descent methods for genus $3$ from~\cite{lr2012}.

\begin{remark}\label{rem:nontrivcov}
  At least in characteristic $0$ and genus $g \leq 2^7$, a covariant $c$ with
  the properties in Proposition~\ref{prop:covdesc} exists. More precisely, if
  we let $f$ be a generic binary form defining $\Xm$, then the covariant form $c
  = (f,f)_{2 g - 2}$ is a nonsingular binary octavic with trivial automorphism
  group.

  Given a genus $g$, this statement is easy to verify with a computer algebra
  package; by the proof of \cite[Prop.2.9]{lrs}, it suffices to produce a single
  example of such an $f$. Usually the first randomly chosen $f$ already works,
  in line with our expectations that a covariant $c$ should generically always
  exist. On the other hand, to prove the existence of such a covariant in the
  generic case for arbitrary genus, let alone for all hyperelliptic curves with
  trivial reduced automorphism group, seems to be more involved.
\end{remark}

\subsubsection{Explicit cocycle construction}

We mention a second approach, which could be used in the unlikely event that no
suitable covariant is available. If $\Xm$ is defined by a binary form $f$ over a
finite Galois extension $M$ of $k$, then one can construct a suitable cocycle
for $\Bm = \Xm / \Aut (\Xm) = \Xm / \iota$ over $M$ by using the fast methods
from~\cite[Sec.2]{lrs}. One then calculates canonical model $\Bm_0$ of $\Bm$
along with the descent morphism $\Bm \to \Bm_0$ as in~\cite{hidrey}. If the
descent obstruction is trivial, then one proceeds as before; one constructs a
descent $\Xm_0$ of $\Xm$ by ramifying over the image of the branch locus of $\Xm
\to \Bm$ under the morphism $\Bm \to \Bm_0$.

\subsection{Counterexamples} \label{sec:counter}

To finish this section, we will show how to obtain explicit counterexamples to
descent. We first treat some classical counterexamples to hyperelliptic descent,
where $K = \CC$ and $k = \RR$. In this case, the classification of the curves
that do not descend is known. These curves were essentially first constructed
by~\cite{bujtur}, but the final correct statement is due to Huggins
in~\cite{huggins-thesis}. The following proposition is a slight improvement of
their results.

\begin{proposition}\label{prop:huggins}
  Let $\Qm_0$ be the pointless conic over $\RR$ defined by the homogeneous
  equation $x^2 + y^2 + z^2 = 0$. Let $(\PP^1 , \rR)$ be one of the divisors
  over $\CC$ defined in~\cite[Prop.5.0.5]{huggins-thesis}. Consider the
  $\CC$-morphism $\varphi: \PP^1 \to \Qm_0$ given by $$(s:t) \to (i (s^2 + t^2)
  : s^2 - t^2 : 2 s t).$$ Then $\rR_0 = \varphi_* (\rR)$ is an $\RR$-rational
  divisor on $\Qm_0$ that defines a hyperelliptic curve $\Xm$ over $\CC$ whose
  field of moduli for the extension $\CC | \RR$ is $\RR$ but which does not
  descend hyperelliptically.

  Up to isomorphism over $\CC$, all counterexamples to hyperelliptic descent
  from $\CC$ to $\RR$ are of the form $\Xm$ considered above. Such a curve $\Xm$
  still descends to $\RR$ if and only if its genus and the cardinality of its
  reduced automorphism group $\Gbar$ of $\Xm$ are both odd.
\end{proposition}
\begin{proof}
  Let $\sigma$ be the generator of $\Gal (\CC | \RR)$. Then $\rR^{\sigma} =
  \alpha_* (\rR)$, where $\alpha : \PP^1 \to \PP^1$ is the $\RR$-rational morphism
  given by $(s:t) \to (-t:s)$. Now we have $\varphi^{\sigma} = \varphi
  \alpha$. Therefore
  \begin{equation*}
    \rR_0^{\sigma} = \varphi_*^{\sigma} (\rR^{\sigma}) = (\varphi \alpha )_*
    (\rR) = \varphi_* (\alpha_* (\rR)) = \varphi_* (\rR) = \rR_0
  \end{equation*}
  and hence $\rR_0$ is indeed $\RR$-rational. All is now clear
  from~\cite[Prop.5.0.5]{huggins-thesis}, except for our sharpening of
  the result (the final line of the proposition). But this follows from
  Theorem~\ref{thm:MainTh4}.
\end{proof}

\begin{remark}
  A signature argument as in~\cite[Prop.4.3]{lr2012} can also be used to prove
  Proposition~\ref{prop:huggins}, except if $\Gbar \cong \CG_{n}$ with $n$ odd
  and either $g/n$ is odd or $(g+1)/n$ is even. Theorem~\ref{thm:MainTh4} shows
  that in these cases, a descent is always possible, and
  Algorithm~\ref{alg:nonhypdesc} shows how this descent can be obtained
  explicitly.
\end{remark}

Having precisely analyzed the obstruction to descent in the previous
subsections, it is now straightforward to give a complete classification of
those hyperelliptic curves with tamely cyclic and nontrivial reduced
automorphism group whose field of moduli is not a field of definition.

\begin{theorem}\label{thm:MainTh5}
  Let $L = k(\sqrt{d_1})$ be a quadratic extension of $k$ defined by an element
  $d_1$ of $k$, and choose $d_2 \in k$ such that $d_2$ is not a norm from $L$.
  Let $u \in L$ be such that $\Nm^{L}_k (u) = 1$. Let $m = 2 \ell$ be
  an even number, and choose $a_m , \ldots , a_0$ in $L$ such that 
  \begin{equation*}
    a_{\ell}^{\sigma} = u a_{\ell} , \, a_{\ell - 1} = u d_2 a_{\ell +
    1}^{\sigma}, \ldots , \, a_0 = u d_2^{\ell} a_m^{\sigma} 
  \end{equation*}
  for the nontrivial element $\sigma$ of $\Gal (L | k)$. Consider the binary
  forms
  \begin{align*}
    f & = a_m x^{m n} + a_{m-1} x^{(m-1)n} z^n + \ldots + a_1 x^n z^{(m-1)n} +
    a_0 z^{m n}  , \\
    g & =  x z f  . 
  \end{align*}
  Suppose that $f$ is of type $(0,n,m)$, so that $g$ is of type $(2,n,m)$, and
  that the geometric automorphism group $\Aut_K (f)$ is generated by $(x,z)
  \mapsto (\zeta_n x,z)$. Then the curves corresponding to $f$ and $g$ have
  field of moduli $k$ for the extension $K | k$ and do not descend
  hyperelliptically to $k$.
  
  Up to isomorphism over $K$, all counterexamples to hyperelliptic descent from
  $K$ to $k$ are of the form $\Xm$ considered above. Such a curve $\Xm$ still
  descends to $k$ if and only if its genus and $n$ are both odd.
\end{theorem}
\begin{proof}
  The forms under consideration are already in normal form. Therefore their
  invariant extension equals the quadratic extension $L$ itself. This makes it
  straightforward to calculate the matrix~\eqref{cocyc4} for these examples,
  which is simply given by $\smat{0}{d_2}{1}{0}$. Combining
  Theorem~\ref{thm:MainTh2} with Proposition~\ref{prop:WeilTest} then shows that
  we indeed get counterexamples.

  The universality statement needs a bit more work. Note first that indeed any
  counterexample is determined by a normal form~\eqref{ccform}-\eqref{cform2}
  in light of Proposition~\ref{prop:f0}(ii). We only have to cull those normal
  forms for which the element $r$ in the matrix~\eqref{cocyc4} is not a norm
  from $L = k (\sqrt{d_1})$. We can do this by inverting the procedure in
  Section~\ref{sec:descent}; one chooses $d_2 = r$ not to be a norm, constructs
  the matrix $A = \smat{0}{d_2}{1}{0}$, and finally determines those forms
  $f$ over $L$ for which there exists a scalar $u$ such that $A f = u
  f^{\sigma}$. This gives the requested forms, with the demand that $\Nm^L_k (u)
  = 1$ coming from the compatibility condition $(f^{\sigma})^{\sigma} = f$.

  The final statement of the Theorem is again a consequence of
  Theorem~\ref{thm:MainTh4}.
\end{proof}

\begin{remark}\label{rem:quatalg}
  Phrased differently, we have shown that the curves constructed in
  Theorem~\ref{thm:MainTh5} descend if and only if the quaternion algebra
  defined by $d_1$ and $d_2$ splits. This gives some unexpected symmetry
  properties for the obstruction, since for example exchanging $d_1$ and $d_2$
  yields the same quaternion algebra.
\end{remark}

%\begin{remark}
%  In the case where $\Xm$ has automorphism group of even order, we can do
%  slightly better, in the sense that the phrase ``Up to isomorphism over $K$''
%  in Theorem~\ref{thm:MainTh5} can then be replaced by ``Up to isomorphism over
%  $L$''. This can be proved as follows.
%
%  Note first that a finite cyclic subgroup of $\PGL_2 (K)$ even cardinality
%  is determined by its unique element $\alpha_0$ of order $2$. This can be shown
%  be a Lie theoretic argument. For a more elementary proof, one first
%  diagonalizes $\alpha_0$ to an element of order $2$, which is then necessarily
%  given by $\smat{-1}{0}{0}{1}$, and which determines the subgroup containing it
%  as being generated by a matrix of the form $\smat{\zeta}{0}{0}{1}$. These are
%  exactly the subgroups used in our normal forms~\eqref{ccform}-\eqref{cform2}.
%  
%  By the results of the previous paragraph, given a curve $\Xm$ of type
%  $(i,n,m)$ with $n$ even over a field $L$, we can diagonalize $\alpha_0$ over
%  $L$ to obtain a normal form of the same type that is still defined over $L$.
%  This implies our claim.
%\end{remark}

\begin{remark}
  For genus $3$, the explicit stratum equations in~\cite[Sec.3]{lr2012} can be
  used to quickly determine whether a given curve $\Xm$ is of a given type
  $(i,n,m)$. For general genus, it is usually easy to verify this once the
  coefficients of the polynomial $f$ defining $\Xm$ are given, by using the
  methods of~\cite[Sec.2]{lrs}.
\end{remark}

This construction gives counterexamples for many more quadratic field
extensions than the usual $\CC | \RR$. Moreover, the cases where $d_2$ \emph{is}
a norm from $L$ yield a host of examples for which it is anything but obvious
that the resulting curves descend, and which we will consider in the next section.

\section{Implementation and examples}\label{sec:imp}

We have used the generic homogeneous dihedral invariants of
Proposition~\ref{prop:GenTInv} in~\texttt{Magma} to give an implementation of
Algorithms~\ref{alg:hypdesc} and~\ref{alg:nonhypdesc} for the curves for which
these invariants suffice. Our code is available online\footnoteref{foothyp}.

The implementation is straightforward considering the constructive methods that
were used. First one determines a generator $\alpha_0$ of the reduced
automorphism group, which can be done effectively by using the methods
in~\cite[Sec.2]{lrs}.  Subsequently, one diagonalizes $\alpha_0$ over an at most
quadratic extension of the base field of $K$. The remaining steps in
Algorithms~\ref{alg:hypdesc} and~\ref{alg:nonhypdesc} (determining and
normalizing the homogeneous invariants, parametrizing to determine a partial
descent, solving a norm equation and if necessary constructing the necessary
cover to define $\Xm$ over $k$) are effective and efficient for `natural' fields
such as number fields and finite fields.

%However, determining the generator $\alpha$ remains difficult in practice. If
%$\# \overline{G} = 2$, then $\alpha$ is defined over the ground field by
%uniqueness, and the methods from~\cite{lrs} can be applied. However, if $n >
%2$, then in the typical case where $k$ is a number field, the number fields
%determined by the cyclic order $n$ subschemes of $\PGL_2 (k)$ have arbitrary
%large discriminant, as one sees by taking Galois twists of $\left<\left(
%\begin{array}{cc} \zeta_n & 0 \\ 0 & 1 \end{array} \right)\right>$. Therefore,
%given $\Xm$, we see no way as yet to specify a finite set of fields containing
%a small splitting for $\Aut_K (X)$. One can evidently take the splitting field
%of the defining polynomial $f$ of $\Xm$ itself, but this is an prohibitively
%expensive computation. Since this is not acute in genus $3$, our main case of
%interest, we mention it as an interesting open problem.

We now give some examples of these computations. Throughout, we will have $k = \QQ$
and $K = \overline{\QQ}$. To begin with, we mention that we usually do not need
the full set of generic homogeneous dihedral invariants in our computations, and
our algorithms take this into account. The following example of a hyperelliptic
curve of type $(0,2,6)$ illustrates this.

\begin{example}\label{ex:nondesc1}
  In Theorem~\ref{thm:MainTh3}, let $d_1 = 2, d_2 = 3$, let $\sigma$ be an
  automorphism of $K$ restricting to a generator of $k (\sqrt{d_1})$, and take
  {\small
  \begin{align*}
    a_6 & = 7 + \sqrt{d_1}, a_5 = 3 - 2 \sqrt{d_1} , a_4 = (1 + \sqrt{d_1}) ,
    a_3 = 12 \sqrt{d_1} , \\ a_2 & = -d_2 a_4^{\sigma} = -d_2 (1 - \sqrt{d_1}),
    a_1 = -d_2^2 a_5^{\sigma} = -d_2^2 (3 + 2 \sqrt{d_1}), a_0 = -d_2^3
    a_6^{\sigma} = -d_2^3 (7 - \sqrt{d_1}).
  \end{align*}
  }
  Let
  \begin{equation*}
    f = a_6 x^{12} + a_5 x^{10} z^{2} + a_4 x^{8} z^{4} + a_3 x^{6} z^{6} + a_2
    x^{4} z^{8} + a_1 x^{2} z^{10} + a_0 z^{12}  .
  \end{equation*}
  As in Remark~\ref{rem:smaller}, the corresponding hyperelliptic curve is
  determined by the following subset of the homogeneous dihedral invariants:
  \begin{equation*}
    (I_1 , I_{2,0} , I_{2,1} , I_{2,2} , I_{3,3,1} , I_{3,3,2} , I_{3,4}) .
  \end{equation*}
  Indeed, one shows by direct calculation that $J_3 \neq J'_3$ for $f$, hence
  also for all its transformations. Having chosen the ordering of the roots $J_3$
  and $J'_3$ of the corresponding quadratic equation, the linear system
  \begin{align*}
    J_i + J'_i & = I_{i,i,1} , \\
    J'_3 J_i + J_3 J'_i  & = I_{3,i}
  \end{align*}
  is always invertible for $i > 3$, determining $J_i$ and $J'_i$ in terms of the
  choice of the order of $J_3$ and $J'_3$ and the invariants $(I_1 , I_{2,0} ,
  I_{2,1} , I_{2,2} , I_{3,3,1} , I_{3,3,2} , I_{3,4})$. In particular, we need
  only normalize these latter invariants to determine the field of moduli
  of our curve. This normalization is
  \begin{equation*}
    \left(1 , \frac{-3\cdot 47}{2^5} , \frac{-1}{2^5} , \frac{1}{2^5 \cdot 3} ,
    \frac{-1}{2^4} , \frac{-1}{2^{15} \cdot 3^2} , \frac{-1}{2^{13} \cdot 3^2}
    \right),
  \end{equation*}
  so the field of moduli is indeed the rational field $k$. Calculating the norm
  obstruction $r$ and $L = \QQ (\sqrt{d_1})$ in Theorem~\ref{thm:MainTh3} and
  using the norm criterion now shows that the curve corresponding to $f$ does
  not descend to $k$. This is as expected, because this example was constructed
  by using Theorem~\ref{thm:MainTh5}.
\end{example}

Now we will consider a hyperelliptic curve of type $(2,2,3)$.
\begin{example}\label{ex:cform2}
  Consider the genus $3$ hyperelliptic curve $\Xm$ over $K$ corresponding to the
  binary form
  \begin{equation*}
    \begin{array}{c}
      (20456 \sqrt{5} + 43640) x^8 + (-17772 \sqrt{5} - 56716) x^7 z + (28984
      \sqrt{5} + 3584) x^6 z^2 \\
       + (25862 \sqrt{5} - 95522) x^5 z^3 + (67320 \sqrt{5} - 136740) x^4 z^4 +
       (84995 \sqrt{5} - 193217) x^3 z^5 \\
       + (75097 \sqrt{5} - 167611) x^2 z^6 + (38764 \sqrt{5} - 86676) x z^7 +
       (7942 \sqrt{5} - 17762) z^8
    \end{array}
  \end{equation*}
  over $L = \QQ(\sqrt{5}) \subset K$. This curve has an automorphism of order $2$,
  and it allows a normal form \eqref{cform2} over $L$ given by
  \begin{align*}
    \begin{array}{c}
      x z ((11270829 \sqrt{5} - 25242007) x^6 + (1408299 \sqrt{5} - 5284449) x^4
      z^2 \\
      + (-5642070 \sqrt{5} - 12929374) x^2 z^4 + (-204992252 \sqrt{5} -
      458411532) z^6)  .
    \end{array}
  \end{align*}
  The normalized homogeneous dihedral invariants of this form now generate the
  field of moduli $k$. They are given by
  \begin{equation*}
    (I_{2,0},I_{2,1},I_{4,4,1},I_{4,4,2}) = \left(\frac{2}{3} , 1 ,
    \frac{29}{3^2} , \frac{2}{3} \right) .
  \end{equation*}
  Lemma~\ref{lem:rgen} shows that the norm obstruction is trivial because $m =
  3$ is odd. In this particular case, this reflects itself in the fact that the
  homogeneous diagonal invariants are themselves already rational. They are
  given by
  \begin{equation*}
    (J_{2,0},J_{2,1},J_4) = \left(\frac{2}{3} , 1 , 3\right)  .
  \end{equation*}
  Reconstructing as in Corollary~\ref{cor:param}, we get the descent
  \begin{equation*}
    y^2 = x z \left(3 x^6 + x^4 z^2 + x^2 z^4 + \frac{2}{9} z^6\right)  .
  \end{equation*}
\end{example}

Finally, we descend a hyperelliptic curve of type $(1,3,3)$.
\begin{example}\label{ex:cform1}
  Consider the genus $4$ hyperelliptic curve $\Xm$ corresponding to the binary
  form
  \begin{align*}
    \begin{array}{c}
      (138076 \sqrt{5} + 291100) x^{10} + (-120728 \sqrt{5} - 370816) x^9 z \\
      + (243042 \sqrt{5} + 208878) x^8 z^2 + (48987 \sqrt{5} - 760529) x^7 z^3
      \\
      + (515947 \sqrt{5} - 751581) x^6 z^4 + (754227 \sqrt{5} - 1880505) x^5 z^5
      \\
      + (1243617 \sqrt{5} - 2713183) x^4 z^6 + (1462433 \sqrt{5} - 3287139) x^3
      z^7 \\
      + (1243263 \sqrt{5} - 2777109) x^2 z^8 + (625402 \sqrt{5} - 1398734) x z^9
      \\
      + (124654 \sqrt{5} - 278722) z^{10}  .
    \end{array}
  \end{align*}
  over $L = \QQ(\sqrt{5})$. This curve has an automorphism of order $3$, and it
  allows a normal form \eqref{cform1} over the ground field given by
  \begin{align*}
    \begin{array}{c}
      z ( (91955817 \sqrt{5} - 213442907) x^9 + (268416746 \sqrt{5} + 589172042)
      x^6 z^3 \\
      + (-30323641593 \sqrt{5} - 67805941509) x^3 z^6 + (3073332514916 \sqrt{5}
      + 6872180416996) z^9)  .
    \end{array}
  \end{align*}
  Lemma~\ref{lem:1nm} shows that the normalized homogeneous diagonal invariants
  for this case will generate the field of moduli $k = \QQ$. In this case, these
  invariants are up to scalar given by
  \begin{equation*}
    (J_{2,0},J_{2,1},J_4) = \left(\frac{2}{3} , 1 , \frac{8}{3^2}\right)  .
  \end{equation*}
  Using the parametrization of Corollary~\ref{cor:param} for the generic case,
  we obtain the following hyperelliptic descent of $\Xm$:
  \begin{equation*}
    y^2 = z \left(\frac{8}{3^2} x^9 + x^6 z^3 + x^3 z^6 + \frac{3}{4} z^9\right)
    .
  \end{equation*}
  As Lemma~\ref{lem:1nm} predicts, this normal form is already defined over the
  field of moduli $k$ itself (rather than over a quadratic extension).
\end{example}

We now discuss some examples of curves of genus $3$. Indeed, this was our
initial motivation for this paper, the cases of genus $2$ having been completely
resolved already in~\cite{mestre} and~\cite{carquer}.

The invariant theory of binary octavics was completely determined by Shioda
in~\cite{shioda67}, and can be applied to solve the descent problem for genus
$3$ hyperelliptic curves with great efficiency. The steps for this are as
follows.
\begin{itemize}
  \item Using the stratum equations from~\cite[Sec.3]{lr2012}, determine the
    geometric automorphism group $G$ of $\Xm$ from its Shioda invariants;
  \item If $G \ncong \DG_4$, then use either the parametrizations or
    reconstruction methods from~\cite[Sec.3]{lr2012} or (in the case $G \cong
    \CG_2^3$) the covariant descent method in~\cite[Sec.3B2]{lrs};
  \item If $G \cong \DG_4$, then determine the homogeneous dihedral invariants of
    $f$, directly or from its Shioda invariants~\footnoteref{footg3}, and apply the methods of this paper.
\end{itemize}

\begin{example}\label{ex:nondesc2}
  As in Example~\ref{ex:nondesc1}, let $d_1 = 2, d_2 = 3$. This time, take
  \begin{align*}
    a_4 &= 7 + \sqrt{d_1}, a_3 = 3 - 2 \sqrt{d_1} , a_2 = 12 \sqrt{d_1} ,\\
    a_1 &= -d_2 a_3^{\sigma} = - d_2 (3 + 2 \sqrt{d_1}), a_0 = -d_2^2
    a_4^{\sigma} = -d_2^2 (7 - \sqrt{d_1}).
  \end{align*}
  Construct the binary octavic
  \begin{equation*}
    f = a_4 x^8 + a_3 x^6 z^2 + a_2 x^4 z^4 + a_1 x^2 z^6 + a_0 z^8 .
  \end{equation*}
  The normalized Shioda invariants of this octavic (and of its transformations
  under $\GL_2 (K)$) are given by
  \begin{align*}
    \begin{array}{c}
      - 5\cdot 7\cdot 401^3 / (3^3 \cdot 13^2 \cdot 23^2 \cdot 1667^2) , \\
      - 5\cdot 7\cdot 401^3 / (3^3 \cdot 13^2 \cdot 23^2 \cdot 1667^2),  \\
      2^4 \cdot 5^4 \cdot 7^{13} \cdot 401^4 \cdot 3435911 / (3^7 \cdot 13^4
      \cdot 23^4 \cdot 1667^4) , \\
      2^3 \cdot 5^4 \cdot 7^{16} \cdot 401^5 \cdot 1663\cdot 29947 / (3^7 \cdot
      13^5 \cdot 23^5 \cdot 1667^5) , \\
      2^3 \cdot 5^7 \cdot 7^{18} \cdot 47\cdot 59\cdot 401^6 \cdot 3271\cdot
      14653 / (3^{11} \cdot 13^6 \cdot 23^6 \cdot 1667^6) , \\
      2^3 \cdot 5^7 \cdot 7^{22} \cdot 401^7 \cdot 166150639393 / (3^{11} \cdot
      13^7 \cdot 23^7 \cdot 1667^7) , \\
      - 2 \cdot 5^7 \cdot 7^{25} \cdot 401^8 \cdot 25309\cdot 148913\cdot 395201
      / (3^{13} \cdot 13^8 \cdot 23^8 \cdot 1667^8) , \\
      2^6 \cdot 5^8 \cdot 7^{27} \cdot 17\cdot 401^9 \cdot 4278649\cdot
      127546933 / (3^{15} \cdot 13^9 \cdot 23^9 \cdot 1667^9) , \\
      - 2^2 \cdot 5^8 \cdot 7^{30} \cdot 11\cdot 61\cdot 401^{10} \cdot
      537787278082528849 / (3^{17} \cdot 13^{10} \cdot 23^{10} \cdot 1667^{10})
      .
    \end{array}
  \end{align*}
  This gives the normalized homogeneous dihedral invariants
  \begin{equation*}
    (I_1 , I_{2,0}, I_{2,1}, I_{3,3,1} , I_{3,3,2}) = \left(1, \frac{-47}{2^5},
    \frac{-1}{2^5 \cdot 3} , \frac{101}{2^6 \cdot 3} , \frac{-47}{2^{15} \cdot
    3^2}\right) ,
  \end{equation*}
  which are somewhat simpler. We have $I_{3,3,1}^2 - 4 I_{3,3,2} = 11^2 13^2 /
  2^{13} 3^2$. This defines the quadratic extension $L = \QQ (\sqrt{2})$ of the
  rational field, which therefore equals the invariant field of $f$ over the
  field of moduli $k = \QQ$. The invariant $I_{2,1}$ is not a norm from this
  extension, so we see by Lemma~\ref{lem:rgen} that no hyperelliptic descent
  exists, and hence no descent at all by Theorem~\ref{thm:MainTh1}.

  We can still use the normalized homogeneous diagonal invariants to get a
  hyperelliptic descent over invariant extension $L$. Up to
  switching $J_3$ and $J'_3$ we have
  \begin{equation*}
    (J_1 , J_{2,0}, J_{2,1}, J_3 ) = \left(1, \frac{-47}{32}, \frac{-1}{96} ,
    \frac{-143 \sqrt{2} + 202}{768} \right) ,
  \end{equation*}
  Using Corollary~\ref{cor:param} in the generic case where the
  parameter is $a_1$, we get the following hyperelliptic descent over $L$:
  \begin{equation*}
    y^2 = (-143/768 \sqrt{2} + 101/384) x^8 - 1/96 x^6 z^2 + x^4 z^4 + x^2 z^6 +
    (1716 \sqrt{2} + 2424) z^8  .
  \end{equation*}
\end{example}

\begin{example}\label{ex:desc1}
  Modifying $d_1 = 3, d_2 = 13$ in Example~\ref{ex:nondesc2} so that the
  obstruction vanishes, we do get a descent to the rationals. Explicitly, this
  descent can be constructed as follows, using Algorithm~\ref{alg:hypdesc}. This
  time the norm obstruction $r$ in~\eqref{cocyc4} equals $144/13$. We then apply
  Proposition~\ref{prop:WeilTest}, taking $\lambda = (-60 - 24 \sqrt{3})/13$ and
  $\beta = 1 / \sqrt{3}$ in the proof.  Transforming the quotient $B = \Xm / G
  \cong \PP^1_K$ into $B_0 \cong \PP^1_k$ by the $N$ from
  Proposition~\ref{prop:WeilTest}, the ramification divisor $\tT$ of $q : Q \to
  B$ transforms to $\tT_0 = [(\sqrt{3} : 1)] +  [(-\sqrt{3} : 1)]$.  Using
  Proposition~\ref{prop:cycP1}, we get the $k$-rational cover $q_0 : Q_0 = \PP^1
  \to \PP^1 = B_0$ given by
  \begin{equation*}
    (x : z) \longmapsto ( x^2 + 3 z^2 : 2 x z )  .
  \end{equation*}
  Under $N$, the divisor $\sS$ on $B$ that is the image on $B$ of branch divisor
  of $\pi_{\iota} : X \to Q$ is mapped from the zero locus of
  \begin{equation*}
    1/5184 (10309 \sqrt{3} + 17745) x^4 + 13/144 x^3 z + x^2 z^2 + x z^3 + (-244
    \sqrt{3} + 420) z^4
  \end{equation*}
  into that of
  \begin{equation*}
    38 x^4 + 320 x^3 z + 657 x^2 z^2 + 924 x z^3 + 387 z^4
  \end{equation*}
  on the canonical model $B_0 = \PP^1_{k}$ of $B$. Taking a suitable
  $k$-rational binary form vanishing on the pullback of this divisor by $q_0$,
  we obtain the hyperelliptic descent
  \begin{align*}
    \begin{array}{c}
      y^2 = 19 x^8 + 320 x^7 z + 1542 x^6 z^2 + 6576 x^5 z^3 + 12006 x^4 z^4 \\
      + 19728 x^3 z^5 + 13878 x^2 z^6 + 8640 x z^7 + 1539 z^8  .
    \end{array}
  \end{align*}
  %as a curve for the normalized homogeneous dihedral invariants
  %\begin{equation}
    %(I_1 , I_{2,0}, I_{2,1}, I_{3,3,1} , I_{3,3,2}) = (1, \frac{-13^2 \cdot
    %23}{2^3 \cdot 3^3} , \frac{13}{2^4 \cdot 3^2} , \frac{3\cdot 1973}{2^5
    %\cdot 3^3} , \frac{-13^4 \cdot 23}{2^{11} \cdot 3^7}) .
  %\end{equation}
\end{example}

\begin{example}\label{ex:condesc}
  Finally, by modifying Example~\ref{ex:nondesc2} to
  \begin{equation*}
    f = a_4 x^{12} + a_3 x^9 z^3 + a_2 x^6 z^6 + a_1 x^3 z^9 + a_0 z^{12} 
  \end{equation*}
  we get a curve that does not descend hyperelliptically but which does descend
  as the cover of a conic. Our implementation of Algorithm~\ref{alg:nonhypdesc}
  returns a divisor on the conic $X^2 - 2 Y^2 + 96 Z^2 = 0$ over which we have
  to branch. The result, whose expression is slightly unwieldy, can be found
  online\footnoteref{foothyp} too; here we just mention that over the finite
  field with $43$ elements, where the hyperelliptic descent obstruction vanishes
  (as over all finite fields by Theorem~\ref{thm:MainTh2}), we obtain the
  descended equation
  \begin{align*}
    y^2 =& \; x^{12} + 25 x^{11} z + 6 x^{10} z^2 + 30 x^9 z^3 + 21 x^8 z^4 + 
    9 x^7 z^5 + 21 x^6 z^6 + \\ & \; 37 x^5 z^7 + 42 x^4 z^8 + 22 x^3 z^9 + 5 x^2
    z^{10} + 37 x z^{11} + 3 z^{12}  .
  \end{align*}
\end{example}

\section{Conclusions and remaining questions}\label{sec:conc}

In~\cite{lr2012} and \cite{lrs}, effective parametrizations of the automorphism
strata in genus $3$ were determined, which return a model over the field of
moduli as long as the reduced automorphism group is not $\CG_2$. These methods
can also be used to obtain equations for the curves with reduced automorphism
group $\CG_2$.  However, these equations can be of degree up to $8$ over the
field of moduli, which is far from optimal. The present work shows how one
calculates whether such a curve admits a (hyperelliptic) descent to the field of
moduli, and how such a descent can be determined explicitly if it exists. Even
if the curve does not descend all the way to the field of moduli, a model over
the quadratic invariant extension of this field can still be constructed
efficiently.

This concludes our explicit arithmetic exploration of the moduli space of
hyperelliptic genus $3$ curves, at least when the characteristic of the ground
field is $0$ or bigger than $7$. 
%We obtain effective methods comparable to the results for genus $2$
%in~\cite{mestre} and~\cite{carquer}; 
Given any tuple of Shioda invariants of a genus $3$ curve, one can now determine
\begin{itemize}
  \item the automorphism group of the curve,
  %\item the field of moduli of the curve,
  \item whether or not the curve descends to the field of moduli, and
  \item a model of the curve over its field of moduli, if it exists.
\end{itemize}
When the characteristic of the ground field is positive and less than or equal
to $7$, a nontrivial effort is already needed to find the appropriate analogues
of the Shioda invariants.

There are some open questions remaining. First of all, though we have given a
complete set of effective methods for determining when the field of moduli is a
field of definition, it remains to descend effectively if the reduced
automorphism group is either not tamely cyclic or trivial. Second, our methods
should apply to the \emph{superelliptic} curves $y^n = f(x,z)$ as well. Third,
it seems likely that the case of hyperelliptic curves in characteristic $2$ will
require completely new methods altogether.

Finally, and most intriguingly, while our perfectness hypothesis on $k$ enables
us to resolve the descent problem for most interesting ground fields (such as
number fields and finite fields), it remain to deal with imperfect base fields
$k$, as mentioned in Remark~\ref{rem:insep}.  Dealing with general ground fields
by further studying the inseparable extension in~\cite{Sekiguchi} seems to merit
a study of its own, not least towards studying the geometric nature of this
extension, which we hope to undertake in the future. Here we merely remark that
by~\cite[Th.1.6.9]{huggins-thesis}, our methods can at least be used to
determine whether or not a descent \emph{exists} in these more general cases,
while a method to explicitly determine a descent still seems to be out of reach.


\begin{thebibliography}{10}

\bibitem{Magma}
W.~Bosma, J.~Cannon, and C.~Playoust.
\newblock The Magma algebra system. I. The user language.
\newblock {\em J. Symbolic Comput.}, 24(3-4):235--265, 1997. Computational
  algebra and number theory (London, 1993).

\bibitem{brandt-stich}
R.~Brandt and H.~Stichtenoth.
\newblock Die {A}utomorphismengruppen hyperelliptischer {K}urven.
\newblock {\em Manuscripta Math.}, 55(1):83--92, 1986.

\bibitem{bujtur}
E.~Bujalance and P.~Turbek.
\newblock Asymmetric and pseudo-symmetric hyperelliptic surfaces.
\newblock {\em Manuscripta Math.}, 108(1):1--11, 2002.

\bibitem{carquer}
G.~Cardona and J.~Quer.
\newblock Field of moduli and field of definition for curves of genus 2.
\newblock {\em Lecture Notes Ser. Comput.}, 13:71--83, 2005.

\bibitem{debes-emsalem}
P.~D{\`ebes} and M.~Emsalem.
\newblock On fields of moduli of curves.
\newblock {\em J. Algebra}, 211(1):42--56, 1999.

\bibitem{earle}
C.~Earle.
\newblock On the moduli of closed {R}iemann surfaces with symmetry.
\newblock {\em Ann. of Math. Studies}, 66:119--130, 1971.

\bibitem{gutsha}
J.~Gutierrez and T.~Shaska.
\newblock Hyperelliptic curves with extra involutions.
\newblock {\em LMS J. Comput. Math.}, 8:102--115, 2005.

\bibitem{hidrey}
R.~A. Hidalgo and S.~Reyes.
\newblock A constructive proof of {W}eil's {G}alois descent theorem.
\newblock Preprint at \url{http://arxiv.org/abs/1203.6294}.

\bibitem{huggins-thesis}
B.~Huggins.
\newblock {\em Fields of moduli and fields of definition of curves}.
\newblock PhD thesis, University of California, Berkeley, Berkeley, California,
  2005.
\newblock \url{http://arxiv.org/abs/math.NT/0610247}.

\bibitem{huggins}
B.~Huggins.
\newblock Fields of moduli of hyperelliptic curves.
\newblock {\em Math. Res. Lett.}, 14(2):249--262, 2007.

\bibitem{lr2012}
R.~Lercier and C.~Ritzenthaler.
\newblock Hyperelliptic curves and their invariants: geometric, arithmetic and
  algorithmic aspects.
\newblock {\em Journal of Algebra}, 372:595--636, Dec. 2012.

\bibitem{lrs}
R.~Lercier, C.~Ritzenthaler, and J.~Sijsling.
\newblock Fast computation of isomorphisms of hyperelliptic curves and explicit
  descent.
\newblock In {\em ANTS X: Proceedings of the Tenth Algorithmic Number Theory
  Symposium}, pages 463--486. Mathematical Science Publishers, 2013.

\bibitem{mestre}
J.-F. Mestre.
\newblock Construction de courbes de genre $2$ \`a partir de leurs modules.
\newblock In {\em Effective methods in algebraic geometry}, volume~94 of {\em
  Prog. Math.}, pages 313--334, Boston, 1991. Birk{\"a}user.

\bibitem{Sekiguchi}
T.~Sekiguchi.
\newblock Erratum: ``On the fields of rationality for curves and for their
  Jacobian varieties'' [Nagoya Math. J. 88 (1982), 197--212; MR0683250
  (85a:14021)].
\newblock {\em Nagoya Math. J.}, 103:163,1986.

\bibitem{serre-coho}
J.P.~Serre.
\newblock {\em Cohomologie galoisienne}, volume 5 of Lecture Notes in
Mathematics. Springer-Verlag, Berlin, fifth edition, 1994.

\bibitem{shimura-moduli}
G.~Shimura.
\newblock On the field of rationality for an abelian variety.
\newblock {\em Nagoya Math. J.}, 45:167--178, 1971.

\bibitem{shioda67}
T.~Shioda.
\newblock On the graded ring of invariants of binary octavics.
\newblock {\em American J. of Math.}, 89(4):1022--1046, 1967.

\bibitem{wehlau}
D.~L. Wehlau.
\newblock Constructive invariant theory for tori.
\newblock {\em Ann. Inst. Fourier (Grenoble)}, 43(4):1055--1066, 1993.

\bibitem{wei56}
A.~Weil.
\newblock The field of definition of a variety.
\newblock {\em American Journal of Mathematics}, 78:509--524, 1956.

\bibitem{xarles-towers}
X.~Xarles
\newblock Trivial points on towers of curves.
\newblock Preprint at \url{http://arxiv.org/abs/1201.2567}.

\end{thebibliography}
\end{document}